\documentclass[a4paper,11pt]{amsart}

\usepackage{amssymb}
\usepackage{stmaryrd}
\usepackage[pagebackref,%colorlinks,linkcolor=black,citecolor=black%
    ,pdfborder={0 0 0}
    ,urlcolor=black,a4paper,hypertexnames=false]{hyperref}
\hypersetup{pdfauthor=Clara L\"oh and Cristina Pagliantini,pdftitle=Integral foliated simplicial volume of hyperbolic $3$-manifolds}

\usepackage{pgf}
\usepackage{tikz}
\usetikzlibrary{decorations.pathmorphing}
\usepackage{tabularx}

\newtheorem{thm}{Theorem}[section]
\newtheorem{prop}[thm]{Proposition}
\newtheorem{lem}[thm]{Lemma}
\newtheorem{cor}[thm]{Corollary}

\newtheorem{question}[thm]{Question}

\theoremstyle{remark}
\newtheorem{rem}[thm]{Remark}
\newtheorem{exa}[thm]{Example}

\newtheorem*{ack}{Acknowledgements}

\theoremstyle{definition}
\newtheorem{defi}[thm]{Definition}
\newtheorem{setup}[thm]{Setup}

\usepackage{amsfonts}
\newcommand{\Z}{\mathbb{Z}}

\newcommand{\R}{\mathbb{R}}
\newcommand{\N}{\mathbb{N}}

\newcommand{\Hyp}{\mathbb{H}}
\newcommand{\Ga}{\Gamma}
\newcommand{\ga}{\gamma}
\newcommand{\La}{\Lambda}
\newcommand{\la}{\lambda}

\DeclareMathOperator{\id}{id}
\DeclareMathOperator{\vol}{vol}
\DeclareMathOperator{\covol}{covol}
\DeclareMathOperator{\res}{res}
\DeclareMathOperator{\Isom}{Isom}

\def\epsilon{\varepsilon}
\def\phi{\varphi}
\def\fa#1{\forall_{#1}\;\;\;}

\newcommand{\ucov}[1]{\widetilde{#1}}
\newcommand{\linfz}[1]{L^\infty(#1,\Z)}

 \makeatletter
 \newcommand\norm{\bBigg@{0.8}}

 \makeatother
 \newcommand{\inparens}[2][flex]{\csname #1l\endcsname(#2%
                                 \csname #1r\endcsname)\mathclose{}}
 \newcommand{\inangles}[2][flex]{\csname #1l\endcsname\langle#2%
                                 \csname #1r\endcsname\rangle\mathclose{}} 
 \newcommand{\innorm}[2][flex]{\csname #1l\endcsname|#2%
                                 \csname #1r\endcsname|\mathclose{}}
 \newcommand{\indnorm}[2][flex]{\csname #1l\endcsname\|#2%
                                 \csname #1r\endcsname\|\mathclose{}}
 \newcommand{\indnorml}[4][flex]{\csname #1l\endcsname\|#2%
                                 \csname #1r\endcsname\|_{#3}^{#4}\mathclose{}}

\newcommand{\lone}[2][flex]{\indnorml[#1]{#2}{1}{}}
\newcommand{\loneR}[2][flex]{\indnorml[#1]{#2}{1}{R}}
\newcommand{\rsv}[2][flex]{\indnorml[#1]{#2}{1}{\R}}
\newcommand{\zsv}[2][flex]{\indnorml[#1]{#2}{1}{\Z}}
\newcommand{\sv}[2][flex]{\indnorm[#1]{#2}}%
\newcommand{\isv}[2][norm]{\indnorml[#1]{#2}{\Z}{}}
  \newcommand{\pfcl}[2][flex]{\csname #1l\endcsname[#2%
                              \csname #1r\endcsname]}
  \newcommand{\ifsv}[2][norm]{\csname #1l\endcsname\bracevert\!#2\!%
                             \csname #1r\endcsname\bracevert}
\newcommand{\stisv}[2][flex]{\indnorml[#1]{#2}{\Z}{\infty}}

\usepackage{color}
\usepackage{pdfcolmk}

\def\draftinfo{}%{\color{red}\ -- \textsf{This is a preliminary version!}}}

\author{Clara L\"oh}
\address{Fakult\"at f\"ur Mathematik\\
         Universit\"at Regensburg\\
         93040~Regensburg\\
         }
\email{clara.loeh@mathematik.uni-r.de}
\author{Cristina Pagliantini}
\address{Department Mathematik\\
         ETH Zentrum\\
         8092~Z\"urich\\
        }
\email{cristina.pagliantini@math.ethz.ch}
\title[Integral foliated simplicial volume of hyperbolic $3$-manifolds]%
      {Integral foliated simplicial volume\\ of hyperbolic $3$-manifolds}
\date{\today.\ \copyright{\ C.~L\"oh, C.~Pagliantini 2013}, \draftinfo\\
     \qquad MSC~2010 classification: 57M27, 57M50, 20F67, 55N99}
\begin{document}

\begin{abstract}
  Integral foliated simplicial volume is a version of simplicial
  volume combining the rigidity of integral coefficients with the
  flexibility of measure spaces. In this article, using the language
  of measure equivalence of groups we prove a proportionality
  principle for integral foliated simplicial volume for aspherical
  manifolds and give refined upper bounds of integral foliated
  simplicial volume in terms of stable integral simplicial
  volume. This allows us to compute the integral foliated simplicial
  volume of hyperbolic \mbox{$3$-mani}\-folds.  This is complemented
  by the calculation of the integral foliated simplicial volume of
  Seifert $3$-manifolds.
\end{abstract}

\maketitle

%%%%%%%%%%%%%%%%%%%%%%%%%%%%%%%%%%%%%%%%%%%%%%%%%%%%%%%%%%%%%%%%%%%%%%%%%%%%%%%%
\section{Introduction}

Integral foliated simplicial volume is a version of simplicial volume
combining the rigidity of integral coefficients with the flexibility
of measure spaces. If $M$ is an oriented closed connected manifold, then
the integral foliated simplicial volume~$\ifsv M$ fits into the sandwich 
\[ \sv M \leq \ifsv M \leq \stisv M, 
\]
where $\sv M$ is the classical simplicial volume~\cite{Gromov} and
where $\stisv M$ is the stable integral simplicial volume (stabilised
over all finite coverings of~$M$). 

Gromov~\cite[p.~305f]{Gromov3} suggested a definition of integral foliated
simplicial volume and an upper estimate of $L^2$-Betti numbers in
terms of integral foliated simplicial volume, which was confirmed by
Schmidt~\cite{mschmidt}. It is an open problem whether integral
foliated simplicial volume coincides with simplicial volume in the
case of aspherical manifolds; an affirmative answer would show that
aspherical manifolds with vanishing simplicial volume have vanishing
Euler characteristic, which is a long-standing open
problem~\cite[p.~232]{Gromov2}.
The only cases in which the integral foliated
simplicial volume has been computed are manifolds that split off an
$S^1$-factor and simply connected
manifolds~\cite[Chapter~5.2]{mschmidt}. Moreover, Sauer proved an
upper bound of a related invariant in terms of minimal
volume~\cite[Section 3]{sauerminvol}.

In the present article, we will prove the following:

\begin{thm}[integral foliated simplicial volume of hyperbolic $3$-manifolds]
  \label{mainthm}
  For all oriented closed connected hyperbolic $3$-manifolds~$M$ the
  integral foliated simplicial volume and the simplicial volume of~$M$
  coincide:
  \[ \ifsv{M} = \sv{M} = \frac{\vol(M)}{v_3}. 
  \]
  Here, $v_3$ denotes the maximal volume of ideal geodesic
  $3$-simplices in~$\Hyp^3$. 
\end{thm}

The second equality in Theorem~\ref{mainthm} is the classic
proportionality principle for hyperbolic manifolds of
Gromov~\cite{Gromov} and Thurston~\cite{Thu}, which also holds 
in more generality\cite{fra,fripa,fujman,BBI,strohm,bucher-karlsson,frigerio,bfs,loeh-sauer}.

The proof of Theorem~\ref{mainthm} consists of the following steps:
Using the language of measure equivalence of groups and techniques of
Bader, Furman and Sauer~\cite{bfs} we show that, similarly to
classical simplicial volume, also integral foliated simplicial volume
of aspherical manifolds satisfies a proportionality principle with
respect to certain parameter spaces (see Section~\ref{sec:pp} and
Section~\ref{sec:ifsv} for the definitions):

\begin{thm}[proportionality principle for integral foliated simplicial volume]
  \label{thm:ppgen}
  Let $M$ and $N$ be oriented closed connected aspherical manifolds of the same dimension 
  satisfying~$\sv M > 0$ and $\sv N > 0$. Suppose that there exists an ergodic 
  bounded  measure equivalence coupling~$(\Omega,\mu)$ of the fundamental groups $\Gamma$ 
  and $\Lambda$ of~$M$ and~$N$, respectively; let $c_\Omega$ be the coupling index of this 
  coupling. 
  \begin{enumerate}
    \item Then 
      \[ \ifsv M ^{\Lambda \setminus \Omega} = c_\Omega \cdot \ifsv N ^{\Gamma \setminus \Omega}. 
      \]
    \item If the coupling~$(\Omega,\mu)$ is mixing, then 
      \[ \ifsv M = c_\Omega \cdot \ifsv N. 
      \]
  \end{enumerate}
\end{thm}

Considering the coupling of uniform hyperbolic lattices given 
by the isometry group, we obtain:

\begin{cor}[a proportionality principle for integral foliated simplicial volume of hyperbolic manifolds]
  \label{cor:pphyp}
  Let $n \in \N$, and let $\Gamma, \Lambda < G := \Isom^+(\Hyp^n)$ be uniform lattices. Then
  \[ \frac{\ifsv{\Hyp^n/\Gamma}^{G/\Lambda}}{\covol(\Gamma)}
     =
     \frac{\ifsv{\Hyp^n/\Lambda}^{G/\Gamma}}{\covol(\Lambda)}
  \]
  and 
  \[ \frac{\ifsv{\Hyp^n/\Gamma}}{\covol(\Gamma)}
     =
     \frac{\ifsv{\Hyp^n/\Lambda}}{\covol(\Lambda)}.
  \]
\end{cor}

Therefore, when estimating integral foliated simplicial volume of a
hyperbolic manifold, we can call other hyperbolic manifolds for
help.

In particular, we obtain the following refinement of the upper bound
of integral foliated simplicial volume in terms of stable integral
simplicial volume:

\begin{cor}[comparing integral foliated simplicial volume and stable 
    integral simplicial volume, hyperbolic case]
  \label{cor:ifsvstisvhyp}
  Let $n \in \N$, and let $M$ and $N$ be oriented closed connected
  hyperbolic $n$-manifolds. Then
  \[ \ifsv M \leq \frac{\vol(M)}{\vol(N)} \cdot \stisv N.
  \]
\end{cor}

Moreover, we exhibit concrete examples of parameter spaces that realize 
stable integral simplicial volume as integral foliated simplicial volume:

\begin{thm}[comparing integral foliated simplicial volume and stable integral simplicial volume, 
  generic case]
  \label{thm:ifsvstisvgen}
  Let $M$ be an oriented closed connected manifold with fundamental group~$\Gamma$ and let $S$ 
  be the set of finite index subgroups of~$\Gamma$. Then
  \[ \stisv M = \ifsv M^{\prod_{\Lambda \in S} \Gamma/\Lambda}. 
  \]
\end{thm}

While Theorem~\ref{thm:ifsvstisvgen} is not necessary to prove
Theorem~\ref{mainthm}, it is of independent interest in the context 
of Question~\ref{q:ifsvstisv}.

As last step in the proof of Theorem~\ref{mainthm}, in dimension~$3$,
we will use the following sequence of hyperbolic manifolds, based on a
variation of a result by Francaviglia, Frigerio and
Martelli~\cite{ffm}:

\begin{thm}[hyperbolic $3$-manifolds with small stable integral simplicial volume]
  \label{thm:hyp3stisv}
  There exists a sequence~$(M_n)_{n \in \N}$ of oriented closed
  connected hyperbolic $3$-manifolds with
  \[ \lim_{n \rightarrow \infty} \frac{\stisv{M_n}}{\sv{M_n}} = 1.
  \]
\end{thm}

Notice that it is still unknown whether stable integral simplicial
volume and simplicial volume coincide for hyperbolic $3$-manifolds.
Francaviglia, Frigerio, and Martelli~\cite{ffm} proved that in all
dimensions bigger than~$3$ a sequence of hyperbolic manifolds as in
Theorem~\ref{thm:hyp3stisv} does \emph{not} exist. Therefore, our
approach does not allow to compute the integral foliated simplicial
volume of higher-dimensional hyperbolic manifolds. Moreover, we do not
know whether integral foliated simplicial volume can be different from
stable integral simplicial volume for aspherical manifolds with enough
finite coverings:

\begin{question}\label{q:ifsvstisv}
  What is the difference between integral foliated simplicial volume and 
  stable integral simplicial volume of aspherical oriented closed connected 
  manifolds with residually finite fundamental group?
\end{question}

\subsection*{Organisation of this article} In Section~\ref{sec:simvol}, we
recall the definition and basic properties of (stable integral)
simplicial volume. In Section~\ref{sec:small3hyp}, we construct
hyperbolic $3$-manifolds with small stable integral simplicial volume,
which proves Theorem~\ref{thm:hyp3stisv}. Section~\ref{sec:ifsv} is an
introduction into integral foliated simplicial volume and basic
operations on parameter spaces. We prove the proportionality
principles Theorem~\ref{thm:ppgen} and Corollary~\ref{cor:pphyp} in
Section~\ref{sec:pp}. Section~\ref{sec:ifsvstisv} is devoted to the
refinements of the comparison between integral foliated simplicial
volume and stable integral simplicial volume and includes a proof of
Theorem~\ref{thm:ifsvstisvgen}.  Finally, in
Section~\ref{sec:mainthm}, we complete the proof of
Corollary~\ref{cor:ifsvstisvhyp} and
Theorem~\ref{mainthm}. Section~\ref{sec:stisvifsv3mnf} contains the
computation of integral foliated simplicial volume of Seifert
manifolds.  

\begin{ack}
  We would like to thank Roman Sauer, Roberto Frigerio, Bruno
  Martelli, and Marco Schmidt for numerous helpful discussions.  In
  particular, we would like to thank Roman Sauer for pointing out a
  mistake in the first version.  This work was partially supported by
  the Graduiertenkolleg \emph{Curvature, Cycles, and Cohomology}
  (Universit\"at Regensburg).  The second author was also partially
  supported by the Swiss National Science Foundation, under the grant
  2000200-144373.
\end{ack}

%\clearpage\clcomm{keeping track of things: should be removed soon ... \tableofcontents\clearpage}

%%%%%%%%%%%%%%%%%%%%%%%%%%%%%%%%%%%%%%%%%%%%%%%%%%%%%%%%%%%%%%%%%%%%%%%%%%%%%%%%
\section{Simplicial volume and (stable) integral simplicial volume}\label{sec:simvol}

In this section we will recall the definition of simplicial volume
introduced by Gromov~\cite{Gromov, loeh}  
and its integral version, which uses integral homology instead of
real homology.

Let $X$ be a topological space. Let
$R$ be a normed ring. 
In this section, we restrict only to the cases $R=\R$ or $\Z$.
For $i\in\N$ we denote by $S_i(X)$ the set of singular $i$-simplices
in $X$, by $C_i(X,R)$ the module
of singular $i$-chains with $R$-coefficients.
The homology of the complex $(C_{\ast} (X,R),\partial_\ast)$,
where $\partial_\ast$ is the usual differential, is 
the singular homology $H_\ast (X,R)$ of $X$ with coefficients
in $R$.

We endow the $R$-module $C_i (X,R)$ with the $\ell^1$-norm defined by
$$
\biggl\|\sum_{\sigma\in S_i(X)} a_\sigma\cdot\sigma \biggr\|_1^R=\sum_{\sigma\in S_i(X)} |a_\sigma|\ ,
$$
where $|\cdot|$ is the norm on $R$.
We denote the norm $\rsv{\cdot}$ simply by $\lone{\cdot}$.
The norm $\loneR{\cdot}$
descends to a semi-norm on $H_\ast (X,R)$, 
which is also denoted by~$\loneR{\cdot}$ and
is defined as follows:
if $\alpha\in H_i (X,R)$, then
$$
\loneR{\alpha}  =  \inf \{\loneR{c} \mid  c\in C_i (X,R),\, \partial c=0,\, 
[c ]=\alpha \} \ .
$$
Note that $\zsv{\cdot}$ on $H_\ast(\cdot,\Z)$ is technically not a semi-norm as 
it is not multiplicative in general (see below).

If $M$ is a closed connected oriented  $n$-manifold,
then we denote the fundamental class of $M$ by $[M]_\Z$, i.e.,~the 
positive generator of $H_n(M,\Z)\cong\Z$. The change 
of coefficients homomorphism $H_n(M,\Z)\longrightarrow H_n(M,\R)$
sends the fundamental class to the real fundamental
class $[M]_\R\in H_n(M,\R)$ of $M$.
The following definition is due to Gromov~\cite{Gromov}:

\begin{defi}[(integral) simplicial volume]
The \emph{simplicial volume} of $M$ is 
$$\sv{M}:= \lone{[M]_\R}\,\in\,\R_{\geq0}.$$ 
The \emph{integral simplicial volume} of $M$
is defined as $\isv{M}:=\zsv{[M]_\Z}\,\in\, \N$.
\end{defi}

Of course we have the inequality $\sv{ M}\leq \isv{M}$ but in general no
equality (for instance, $\sv[norm]{S^1}=0$ but $\isv{S^1}=1$).
The integral simplicial volume 
does not behave as nicely as the simplicial volume.
For example,
it follows from the definition that $\isv{M}\geq 1$ for every manifold $M$. 
Therefore, the integral simplicial volume cannot be multiplicative with respect 
to finite coverings (otherwise
it would vanish on manifolds that admit finite non-trivial self-coverings,
such as $S^1$). 
Moreover, as we mentioned before, the $\ell^1$-semi-norm on integral 
homology is not really a semi-norm, since the equality
$\zsv{n\cdot \alpha}=|n|\cdot \zsv{\alpha}$, may not hold for all 
$\alpha\in H_\ast(X;\Z)$ and all~$n\in \Z$. 
Indeed, it is easy to see that 
$\zsv{ n\cdot [S^1]_\Z}=1$ for every $n\in\Z\setminus\{0\}$.

We may consider a \emph{stable} version of the integral 
simplicial volume:% $\stisv{M}$:

\begin{defi}[stable integral simplicial volume]
The \emph{stable integral simplicial volume} of an oriented closed
connected manifold~$M$ is
$$
\stisv{M}:=\inf\Bigl\{\frac{1}{d}\cdot \isv{\overline{M}}\,\Bigm{|}
\, d\in\N,\,\, \text{there is a $d$-sheeted covering} \,\,
\overline{M}\rightarrow M\Bigr\}.
$$
\end{defi}

Since the simplicial volume is multiplicative under finite coverings
\cite[Proposition 4.1]{loeh}, it is  clear that
$\sv{M}\leq \stisv{M}$, but in general they are not equal:
\begin{thm}[{\cite[Theorem 2.1]{ffm}}]
For every $n \in \N_{\geq 4}$ there exists a constant~$C_n<1$ such that the following holds:
Let $M$ be an oriented closed connected hyperbolic $n$-manifold. Then 
$$\sv{M}\leq C_n\cdot\stisv{M}.$$
\end{thm}

For hyperbolic $3$-manifolds it is still an open question 
whether the simplicial volume
and the stable integral simplicial volume are the same:
our Theorem~\ref{thm:hyp3stisv} gives a partial answer.

%%%%%%%%%%%%%%%%%%%%%%%%%%%%%%%%%%%%%%%%%%%%%%%%%%%%%%%%%%%%%%%%%%%%%%%%%%%%%%%%
\section{Hyperbolic $3$-manifolds\\ with small stable integral simplicial volume}
\label{sec:small3hyp}
 
In this section we will prove Theorem~\ref{thm:hyp3stisv} 
following an  argument of Francaviglia, Frigerio, and Martelli
\cite[Corollary 5.16]{ffm}. For the sake of completeness, we recall
some background on triangulations and special complexity. 

\begin{defi}[triangulation]
A \emph{triangulation} of a closed $3$-manifold~$M$ is a realization
of the manifold~$M$ as the gluing of finitely many tetrahedra via some
simplicial pairing of their faces.  Turning to the case of a compact
manifold $M$ with non-empty boundary $\partial M$, one can adapt the
notion of triangulation: an \emph{ideal triangulation} of $M$ is a
decomposition of its interior~$\text{Int}(M)$ into tetrahedra with
their vertices removed.  An (ideal) triangulation is
\emph{semi-simplicial} if all the edges have distinct vertices.
\end{defi}
\begin{defi}[special complexity]
Let $M$ be a compact $3$-manifold, possibly with boundary. The 
\emph{special complexity} $c_S(M)$ of $M$ is 
the minimal number of vertices in a 
special spine for $M$. 
\end{defi}
We refer to the works of Matveev \cite{Matv, Matv2}
for the definition of special spines and their properties. 
For our purpose we just need to recall that a special
spine is dual to a triangulation.
In particular, there is a bijection between the simplices in a triangulation
and the vertices in the dual special 
spine~\cite[Theorem~1.1.26 and Corollary~1.1.27]{Matv2}.
With an abuse of notation we call the \emph{true} vertices of a
special spine in Matveev's definition simply vertices of a special spine.

\begin{thm}[Matveev {\cite[Corollary~1.1.28]{Matv2}}]
\label{Matveev}
Let $M$ be a compact $3$-manifold whose interior $\text{Int}(M)$
admits a complete finite volume hyperbolic structure. 
Then there is a bijection between special spines and ideal
triangulations of $M$ such that the number of vertices in the special spine
is equal to the number of tetrahedra in the corresponding triangulation.
\end{thm}

\begin{rem}\label{complexity}
Matveev~\cite{Matv} introduced the more general notion of complexity, 
which involves spines that are not necessarily dual to triangulations
and showed that complexity and special complexity are equal for any closed
orientable irreducible $3$-manifold distinct from 
$S^3$, $\mathbb{RP}^3$, and $L(3,1)$.
\end{rem}
\begin{rem}
Matveev~\cite{Matv} and Martelli~\cite{mar} used two slighly different
definitions of spines but the two notions are equivalent and lead
to the same complexity \cite[Section 7]{mar}. Coherently with 
\cite{ffm} we use Martelli's definition.
\end{rem}

\begin{rem}\label{semi-simplicial}
Let $M$ be an oriented closed connected manifold and $T$ be a semi-simplicial
triangulation of $M$.
Fixing an order of the vertices and 
a suitable choice of orientation-preserving parametrisations 
$\sigma_1,\dots,\sigma_k$ of simplices of
$T$, then $\sigma_1+\dots+\sigma_k$
represents the integral 
fundamental class of $M$.
Therefore the number of tetrahedra in a semi-simplicial triangulation
provides an upper bound for the integral simplicial volume of $M$.
\end{rem}
In order to adapt the proof of Francaviglia, Frigerio and 
Martelli~\cite[Corollary 5.16]{ffm}
we need the following results:
\begin{prop}\label{prop1}
Let $\overline{M_{(5)}}$ be the compactification of the $5$-chain link complement 
$M_{(5)}$. Then we have: 
$$c_S(\overline{M_{(5)}})= 10 = \sv{\overline{M_{(5)}}}.$$
Moreover, the value of $c_S(\overline{M_{(5)}})$ 
is realized by a special spine dual to a 
semi-simplicial triangulation. 
\end{prop}
\begin{proof}
The $5$-chain link complement has a hyperbolic structure 
\cite{NR} and admits an ideal triangulation  
with $10$ ideal 
and regular tetrahedra
such that each edge has vertices in different cusps
\cite[Section 5.2]{gordon}.  
By the proportionality between simplicial volume and Riemannian
volume (which holds both in the compact case~\cite{Gromov,Thu} and in 
the cusped case~\cite{fra,fripa,fujman,BBI}) we have 
$$\sv{\overline{M_{(5)}}}=\frac{\vol(M_{(5)})}{v_3}=10.$$
Moreover, Theorem~\ref{Matveev} implies
$$c_S(\overline{M_{(5)}})\leq 10. $$

The equality follows from the fact that for every oriented connected  
finite volume hyperbolic $3$-manifold $M$ with compactification~$\overline M$ 
the inequality $\sv{\overline{M}}\leq c_S(\overline{M})$ holds.
Indeed, an argument by Francaviglia~\cite[Theorem 1.2 and Proposition 3.8]{fra}
guarantees that the volume of $M$ can be computed by straightening 
any ideal triangulation of $M$ and then summing the volume of the
straight version of the tetrahedra.
\end{proof}

\begin{prop}\label{prop2}
Let $N$ be the compactification of a finite volume oriented connected
hyperbolic $3$-manifold and suppose that $N$ admits a semi-simplicial
triangulation that realizes the value of $c_S(N)$.
Let $M$ be a manifold obtained by Dehn filling on~$N$.
Then 
$$
\stisv{M}\leq c_S(N).
$$
\end{prop}
\begin{proof}
As pointed out in Remark~\ref{semi-simplicial} we 
estimate the integral simplicial volume
by the number of vertices of the special spine
dual to a semi-simplicial triangulation.

From the special spine $P$ dual to a semi-simplicial 
triangulation of $N$ that realizes $c_S(N)$,
we construct a special spine for $M$, and hence for its finite coverings, 
such that the associated 
triangulations are still semi-simplicial. 
We obtain $\stisv{M}\leq c_S(N)$ following step by step
the argument of Francaviglia, Frigerio, and 
Martelli~\cite[Proposition 5.15]{ffm}.

More precisely:
Let $T_1,\dots,T_k$ be the boundary tori of~$N$. For every $i\in\{1,\dots, k\}$
let $V_i$ be an open solid torus in $M\backslash P$ created 
by Dehn filling on the boundary component $T_i$.
Let $D^1_i$ and $D^2_i$ be a pair of parallel meridian discs
of $V_i$. If $D^1_i$ and $D^2_i$ are generic with respect
to the cellularization induced by $P$ on $T_i$ \cite[Lemma 5.9]{ffm},
the spine $P\cup D^1_i\cup D^2_i$ is special, dual to
a semi-simplicial triangulation, and with $v_\partial^{T_i}$ vertices added
to the ones of $P$.
Gluing a pair of parallel discs for each boundary torus $T_i$ we obtain a 
special spine $Q=P\cup \bigcup_{i=1}^k(D^1_i\cup D^2_i)$ for $M$ with 
$c_S(N)+\sum_{i=1}^k v^{T_i}_\partial+v_I$ vertices, where 
$v_I$ is the number of vertices created by 
intersections between discs added in different boundary components.

Since $\pi_1(M)$ is residually finite, for every $n>0$ there exist $n_0>n$, 
$h>0$ and a regular covering $p:\overline{M}\rightarrow M$ of degree $hn_0$
such that, for every $i\in\{1,\dots,k\}$, the preimage $p^{-1}(V_i)$ consists of $h$ open
solid tori $\overline{V}_i^1,\dots,\overline{V}^h_i$
each winding $n_0$ times along $V_i$ via $p$.
The special spine $Q$ of $M$ lifts to a special spine $\overline Q := p^{-1}(Q)$
 of $\overline{M}$. In particular, each pair of parallel discs added to $P$
lifts to $n_0$ copies of pairs of parallel discs in each open 
solid torus~$\overline{V}_i^j$. 
Removing $2n_0-2$ discs for each open solid torus in $\overline{Q}$, 
we obtain again a special spine $\overline{Q}'$ 
dual to a semi-simplicial triangulation of $\overline{M}$.

By Remark~\ref{semi-simplicial} 
we now estimate $\isv{\overline{M}}$ with the number
of vertices of  $\overline{Q}'$:
$$\stisv{M}\leq \frac{\isv{\overline{M}}}{hn_0}\leq 
c_S(N)+\frac{\sum_{i=1}^k v^{T_i}_\partial+v_I}{n_0}.$$
Since this holds for every $n>0$ and since $n_0 > n$, we get the conclusion.
\end{proof}
We can now complete the proof of Theorem~\ref{thm:hyp3stisv}:
\begin{proof}[Proof of Theorem~\ref{thm:hyp3stisv}]
Let $(M_n)_{n\in \N}$ be a family of hyperbolic $3$-manifolds obtained by Dehn 
filling on $\overline{M_{(5)}}$. 
By Thurston's Dehn filling Theorem \cite[Chapter 5, page 118]{Thu} it follows 
that $\lim_{n\rightarrow \infty}\vol(M_n)=\vol(M_{(5)})$, 
which implies $\lim_{n\rightarrow \infty}\sv{M_n}= \sv{\overline{M_{(5)}}}$
using the proportionality
principle for hyperbolic manifolds
(which holds both in the compact case~\cite{Gromov,Thu} and in 
the cusped case~\cite{fra,fripa,fujman,BBI})
Then we have:
$$1\leq\frac{\stisv{M_n}}{\sv{M_n}}\leq\frac{c_S(\overline {M_{(5)}}) }
{\sv{M_n}}\stackrel{n\rightarrow \infty}{\longrightarrow}
\frac{c_S(\overline {M_{(5)}}) }{\sv{\overline{M_{(5)}}}}=1.$$
where the second inequality follows by Proposition \ref{prop2}, 
and the last equality by Proposition \ref{prop1}.
\end{proof}

%%%%%%%%%%%%%%%%%%%%%%%%%%%%%%%%%%%%%%%%%%%%%%%%%%%%%%%%%%%%%%%%%%%%%%%%%%%%%%%%
\section{Integral foliated simplicial volume}\label{sec:ifsv}

In the following, we will recall the precise definition of integral
foliated simplicial volume by Schmidt~\cite{mschmidt} and discuss
basic facts about the effect of changing parameter spaces.

%%%%%%%%%%%%%%%%%%%%%%%%%%%%%%%%%%%%
\subsection{Definition of integral foliated simplicial volume}

Integral foliated simplicial volume is a version of simplicial volume
combining the rigidity of integral coefficients with the flexibility
of measure spaces. More precisely, integral foliated simplicial volume
is defined via homology with twisted coefficients in function spaces
of probability spaces that carry an action of the fundamental
group. Background on the convenient category of standard Borel spaces 
can be found in the book by Kechris~\cite{kechris}.

\begin{defi}[parametrised fundamental cycles]\label{def:pfc}
  Let $M$ be an oriented closed connected $n$-manifold with
  fundamental group~$\Gamma$ and universal covering~$\ucov M
  \longrightarrow M$. 
  \begin{itemize}
    \item A \emph{standard Borel space} is a measurable space that 
      is isomorphic to a Polish space with its Borel $\sigma$-algebra. 
      A \emph{standard Borel probability space} is a standard Borel 
      space together with a probability measure.
    \item A \emph{standard $\Gamma$-space} is a standard Borel
      probability space~$(X,\mu)$ together with a measurable
      $\mu$-preserving (left) $\Gamma$-action. If the probability 
      measure is clear from the context, we will abbreviate~$(X,\mu)$ 
      by~$X$.
    \item If $(X,\mu)$ is a standard $\Gamma$-space, then we
      equip~$\linfz {X,\mu}$ with the right $\Gamma$-action 
      \begin{align*}
        \linfz {X,\mu} \times \Gamma & \longrightarrow \linfz{X,\mu}\\
        (f,g) & \longmapsto \bigl(x \mapsto (f \cdot g)(x) := f(g \cdot x)\bigr).
      \end{align*}
      and we write $i_M^X$ for the change of coefficients homomorphism
      \begin{align*}
        i_M^X \colon C_*(M,\Z) \cong \Z \otimes_{\Z \Gamma} C_*({\ucov M},\Z)
        & \longrightarrow \linfz X \otimes_{\Z \Gamma} 
        C_*({\ucov M},\Z)\\
        1 \otimes c 
        & \longmapsto 1 \otimes c
      \end{align*}
      induced by the inclusion~$\Z \hookrightarrow \linfz X$ as
      constant functions.
    \item If $(X,\mu)$ is a standard $\Gamma$-space, then 
      \begin{align*} 
        \pfcl M  ^X := H_n(i^X_M)([M]_\Z) & \in
          H_n\bigl(M, \linfz X\bigr) \\
        &= H_n\bigl(\linfz X \otimes_{\Z \Gamma} C_*({\ucov M},\Z)
        \bigr)
        \end{align*}
        is the \emph{$X$-parametrised fundamental class of~$M$}. All
        cycles in the chain complex~$C_*(M,\linfz X) = \linfz X \otimes_{\Z\Gamma} C_*(\ucov M,\Z)$ 
        representing~$\pfcl M^X$ are called
        \emph{$X$-parametrised fundamental cycles of~$M$}.
  \end{itemize}
\end{defi}

The integral foliated simplicial volume is now defined as the infimum 
of $\ell^1$-norms over all parametrised fundamental cycles:

\begin{defi}[integral foliated simplicial volume]
  Let $M$ be an oriented closed connected $n$-manifold with 
  fundamental group~$\Gamma$, and let $(X,\mu)$ be a standard 
  $\Gamma$-space.
  \begin{itemize}
   \item Let $\sum_{j=1}^k f_j\otimes \sigma_j \in C_*\bigl(M,
     \linfz X\bigr)$ be a chain in \emph{reduced form}, i.e., the
     singular simplices~$\sigma_1, \dots, \sigma_k$ on~$\ucov M$ satisfy
     $\pi\circ \sigma_j \neq \pi\circ \sigma_\ell$ for all~$j, \ell
     \in \{1,\dots,k\}$ with $j \neq \ell$ (where $\pi \colon \ucov M
     \longrightarrow M$ is the universal covering map). Then we define
         \[    \ifsv[bigg]{\sum_{j=1}^k f_j \otimes \sigma_j}^X
            := \sum_{j=1}^k \int_X |f_j| \,d\mu \in \R_{\geq 0}.
         \]
   \item The
         \emph{$X$-parametrised simplicial volume of~$M$},
         denoted by~$\ifsv M ^X$, is the infimum of the $\ell^1$-norms
         of all $X$-parametrised fundamental cycles of~$M$.  
   \item The
         \emph{integral foliated simplicial volume of~$M$}, denoted 
         by~$\ifsv M$, is the infimum of all~$\ifsv M ^X$ over all 
         isomorphism classes of standard $\Gamma$-spaces~$X$.  
  \end{itemize}
\end{defi}

\begin{rem}
  Let $\Gamma$ be a countable group.  The class of isomorphism classes
  of standard $\Gamma$-spaces indeed forms a set~\cite[Remark~5.26]{mschmidt}.
\end{rem}

\begin{rem}
  Schmidt's original definition~\cite[Definition~5.25]{mschmidt}
  requires the actions of the fundamental group on the parameter
  spaces to be essentially free. However, allowing also parameter
  spaces with actions that are not essentially free does not change
  the infimum (Corollary~\ref{cor:infpar}).
\end{rem}

\begin{exa}[trivial parameter space]\label{exa:trivialpar}
  Let $M$ be an oriented closed connected manifold with fundamental group~$\Gamma$. 
  If $(X,\mu)$ is a standard $\Gamma$-space consisting of a single point, then 
  $\linfz X \cong \Z$ (as $\Z \Gamma$-modules with trivial $\Gamma$-action) and so
  \[ \ifsv M^X = \isv M.
  \]
  More generally, in combination with Proposition~\ref{prop:prodpar}~(2), we obtain: 
  If $(X,\mu)$ is a standard $\Gamma$-space with trivial $\Gamma$-action, then
  \[ \ifsv M^X = \isv M. 
  \]
  In particular, if $M$ is simply connected, then $\ifsv M ^X = \isv M$ for 
  all standard Borel probability spaces~$(X,\mu)$~\cite[Proposition~5.29]{mschmidt}.
\end{exa}

\begin{prop}[comparison with (integral) simplicial volume~\protect{\cite[Remark~5.23]{mschmidt}}]\label{prop:compisv}
  Let $M$ be an oriented closed connected $n$-manifold with
  fundamental group~$\Gamma$, and let $(X,\mu)$ be a standard
  $\Gamma$-space. Then
  \[ \sv M \leq \ifsv M ^X \leq \isv M. 
  \]
\end{prop}
\begin{proof}
  The linear map $\linfz X \longrightarrow \R$ given by integration
  with respect to~$\mu$ maps the constant function~$1$ to~$1$ and is
  norm-non-increasing with respect to the $\ell^1$-norm on~$\linfz X$.
  From the first property, we easily deduce that the induced
  map~$C_n(M,\linfz X) \longrightarrow C_n(M,\R)$ maps
  \mbox{$X$-para}\-metrised fundamental cycles to $\R$-fundamental cycles, and
  so
  \[ \sv M \leq \ifsv M ^X. 
  \]

  The inclusion~$\Z \hookrightarrow \linfz X$ as constant
  functions is isometric with respect to the $\ell^1$-norm and the
  induced map~$C_n(M,\Z) \longrightarrow C_n(M,\linfz X)$ maps
  fundamental cycles to $X$-parametrised fundamental cycles. Hence, 
  \[ \ifsv M ^X \leq \isv M. \qedhere
  \]
\end{proof}

\begin{rem}[real coefficients]\label{rem:compisvreal}
  Arguments analogous to the ones in the proof of the previous
  proposition show that
  \[ \ifsv M^{L^\infty(X,\mu,\R)} = \sv M 
  \]
  holds for all oriented closed connected manifolds~$M$ and all 
  standard $\pi_1(M)$-spaces~$(X,\mu)$. Here, $\ifsv M^{L^\infty(X,\mu,\R)}$ 
  denotes the number defined like~$\ifsv M^X$, but using~$L^\infty(X,\mu,\R)$ 
  instead of~$\linfz{X,\mu}$.
\end{rem}

We recall two ``indecomposability'' notions for parameter spaces from
ergodic theory:

\begin{defi}[ergodic/mixing parameter spaces]\label{def:erg/mix}
  Let $\Gamma$ be a countable group. 
  \begin{itemize}
    \item A standard $\Gamma$-space~$(X,\mu)$ is \emph{ergodic} 
      if every $\Gamma$-invariant measurable subset~$A \subset X$ 
      satisfies~$\mu(A) \in \{0,1\}$ (equivalently, $\linfz {X,\mu}^\Gamma$ 
      contains only the constant functions).
    \item A standard $\Gamma$-space~$(X,\mu)$ is called \emph{mixing} if for
      all measurable subsets~$A, B \subset X$ and all
      sequences~$(g_n)_{n \in \N}$ in~$\Gamma$ with
      $\lim_{n\rightarrow \infty} g_n = \infty$ we have
      \[ \lim_{n\rightarrow \infty} \mu(A \cap g_n \cdot B) 
         = \mu(A) \cdot \mu(B).
      \]
      Here, $\lim_{n\rightarrow \infty} g_n = \infty$ means that the
      sequence~$(g_n)_{n\in\N}$ eventually leaves any finite subset
      of~$\Gamma$, i.e., that for all finite subsets~$F \subset
      \Gamma$ there is an~$N \in \N$ such that for all~$n \in \N_{\geq
        N}$ we have~$g_n \in \Gamma \setminus F$.
  \end{itemize}
\end{defi}

Clearly, for infinite discrete groups, any mixing parameter space is
also ergodic. Moreover, any countably infinite group admits an
essentially free mixing parameter space:

\begin{exa}[Bernoulli shift]\label{exa:bernoullishift}
  The \emph{Bernoulli shift} of a countable group~$\Gamma$ is the
  standard Borel space~$(\{0,1\}^\Gamma, \bigotimes_\Gamma (1/2 \cdot
  \delta_0 + 1/2 \cdot \delta_1)$, endowed with the translation
  action. If $\Gamma$ is infinite, this standard $\Gamma$-space is
  essentially free and mixing (and hence
  ergodic)~\cite[Lemma~3.37]{mschmidt}. 
\end{exa}

For ergodic parameter spaces, the parametrised fundamental class indeed 
is a generator of the corresponding top homology with twisted coefficients:

\begin{rem}
  Let $M$ be an oriented closed connected $n$-manifold with
  fundamental group~$\Gamma$, and let $(X,\mu)$ be an ergodic standard
  $\Gamma$-space. Then the inclusion~$\Z \hookrightarrow \linfz X
  ^\Gamma$ as constant functions is an isomorphism, and so the 
  change of coefficients homomorphism 
  \[ H_n(i_M^X) \colon \Z \cong H_n(M,\Z) \longrightarrow 
     H_n(M, \linfz X) \cong \linfz X ^\Gamma \cong \Z
  \]
  is an isomorphism; the isomorphism~$H_n(M,\linfz X) \cong \linfz
  X^\Gamma$ is a consequence of Poincar\'e duality with twisted
  coefficients~\cite[Theorem~2.1, p.~23]{wall}.  
\end{rem}

Furthermore, we will see that ergodic parameters suffice to describe 
the integral foliated simplicial volume (Proposition~\ref{prop:ergpar}). 

\begin{rem}[lack of functoriality]
  Ordinary simplicial volume has the following functoriality property:
  If $f \colon M \longrightarrow N$ is a continuous map betweeen
  oriented closed connected manifolds of the same dimension of
  degree~$d$, then
  \[  |d| \cdot \|N\| = \bigl\| d \cdot [N]_\R \bigr\|_1 =
  \|H_*(f,\R)([M]_\R)\|_1\leq \|M\|.
  \]
  However, when dealing with integral coefficients, the first equality
  might fail in general (because we will not be able to divide
  representatives of~$d \cdot [N]_\Z$ by~$d$). Therefore, integral
  simplicial volume, stable integral simplicial volume and integral
  foliated simplicial volume suffer from a lack of good estimates in 
  terms of mapping degrees.
\end{rem}

In the following, we will investigate some of the effects of changing
parameter spaces. To this end, we will use the following
comparison mechanism:

\begin{prop}[comparing parameter spaces]\label{prop:genprinc}
  Let $M$ be an oriented closed connected $n$-manifold with
  fundamental group~$\Gamma$, let $(X,\mu)$ and $(Y,\nu)$ be standard
  $\Gamma$-spaces, and let $\varphi \colon X \longrightarrow Y$ be a
  measurable $\Gamma$-map. Moreover, suppose that
  \begin{align} 
    \mu \bigl( \varphi^{-1}(A)\bigr) \leq \nu(A) \label{measurecompatibility}
  \end{align}
  holds for all measurable sets~$A \subset Y$.
  Then 
  \[ \ifsv M ^{X} \leq \ifsv M^Y. 
  \]
\end{prop}
\begin{proof}
  We consider the (well-defined) chain map
  $\Phi := \linfz \varphi \otimes_{\Z \Gamma} \id_{C_*(\ucov M,\Z)}$:
  \begin{align*}
    \Phi
    \colon 
    \linfz Y \otimes_{\Z\Gamma} C_*(\ucov M,\Z) 
    & \longrightarrow
    \linfz X \otimes_{\Z\Gamma} C_*(\ucov M,\Z)
    \\
    f \otimes \sigma 
    & \longmapsto f \circ \varphi \otimes \sigma.
  \end{align*}
  In view of the compatibility of~$\varphi$ with 
  the measures (Equation~\eqref{measurecompatibility}), we see that
  \[ \ifsv[big]{\Phi(c)}^X \leq \ifsv c ^Y
  \]
  holds for all chains~$c \in \linfz Y \otimes_{\Z\Gamma} C_*(\ucov M,\Z)$. 
  Moreover, Equation~\eqref{measurecompatibility} shows that 
  $\varphi$ is $\nu$-almost surjective. In particular, $\linfz\varphi$ 
  maps $\nu$-almost constant functions to $\mu$-almost constant functions (with the same value).
  From this we can easily conclude that $\Phi$ maps $Y$-parametrised 
  fundamental cycles to $X$-parametrised fundamental cycles. Taking 
  the infimum over all $Y$-parametrised fundamental cycles of~$M$ thus 
  leads to
  $\ifsv M ^X \leq \ifsv M^Y$.
\end{proof}

We will now consider products, convex combinations, ergodic
decomposition and induction/restriction of parameter spaces.

%%%%%%%%%%%%%%%%%%%%%%%%%%%%%%%%%%%%
\subsection{Products of parameter spaces} % and inverse limits?

\begin{prop}[products of parameter spaces]\label{prop:prodpar}
  Let $M$ be an oriented closed connected $n$-manifold with
  fundamental group~$\Gamma$.
  \begin{enumerate}
    \item Let $I$ be a non-empty, countable (or finite) set. If 
          $(X_i, \mu_i)_{i \in I}$ is a family of standard $\Gamma$-spaces, 
          then also the product
          \[ (Z, \zeta) := \Bigl( \prod_{i \in I} X_i,  
                                  \bigotimes_{i \in I} \mu_i
                           \Bigr), \]
          equipped with the diagonal $\Gamma$-action, is a standard $\Gamma$-space, 
          and 
          \[ \ifsv M ^Z \leq \inf_{i \in I} \ifsv M ^{X_i}. 
          \]
    \item Let $(X,\mu)$ be a standard $\Gamma$-space and let $(Y,\nu)$ be
          some standard Borel probability space. Then
          \begin{align*}
            \ifsv M ^Z & = \ifsv M ^X,
          \end{align*}
          where $Z := X \times Y$ is given the $\Gamma$-action induced by
          the $\Gamma$-action on~$X$ and where $\zeta := \mu \otimes \nu$
          is the product measure on~$Z$.    
  \end{enumerate}
\end{prop} 
\begin{proof}
  The first part follows by applying Proposition~\ref{prop:genprinc} 
  for all~$i \in I$ to the projection~$\prod_{j \in I} X_j \longrightarrow X_i$.

  We now show the second part (following a similar argument by
  Schmidt~\cite[Proposition~5.29]{mschmidt}): We can view~$(Y,\nu)$ as
  standard $\Gamma$-space with trivial $\Gamma$-action. Then we obtain
  \[ \ifsv M ^Z \leq \ifsv M^X 
  \]
  from the first part. For the converse inequality, we consider a $Z$-parametrised 
  fundamental cycle~$c = \sum_{j=0}^k f_j \otimes \sigma_j \in \linfz Z
  \otimes_{\Z \Gamma} C_n({\ucov M},\Z)$ in reduced form. So, if $c_\Z \in \Z \otimes_{\Z
  \Gamma} C_n({\ucov M},\Z)$ is a fundamental cycle of~$M$, there
  is a chain~$d \in \linfz Z \otimes_{\Z \Gamma} C_{n+1}({\ucov
  M},\Z)$ such that
  \[ c - c_\Z = \partial d 
     \in \linfz Z \otimes_{\Z \Gamma} C_n({\ucov M},\Z). 
  \]
  Therefore, for $\nu$-almost all~$y \in Y$, the chain
  \[ c_y := \sum_{j=0}^k \bigl( x \mapsto f_j(x,y)
                         \bigr)\otimes \sigma_j 
     \in \linfz X \otimes_{\Z \Gamma} C_n({\ucov M},\Z) 
  \]
  is well-defined ($\Gamma$ acts trivially on~$Y$) and an
  $X$-parametrised fundamental cycle in reduced form (witnessed by the
  corresponding evaluation of~$d$ at~$y$).  By Fubini's theorem,
  \begin{align*}
        \ifsv c ^Z 
%    & = \sum_{j=0}^k \int_{Z} |f_j| \,d\zeta \\
    & = \int_{X \times Y} \sum_{j=0}^k |f_j| \,d(\mu \otimes \nu)
      = \int_Y \int_X \sum_{j=0}^k |f_j(x,y)| \,d\mu(x) \,d\nu(y)\\
    & = \int_Y \ifsv {c_y}^X \,d\nu(y).
  \end{align*}
  Hence, there is a~$y \in Y$ such that $c_y$ is an $X$-parametrised fundamental cycle 
  and 
  $\ifsv{c_y}^X \leq \ifsv{c}^Z$.
  Taking the
  infimum over all $Z$-parametrised fundamental cycles~$c$ shows
  $\ifsv M ^X \leq \ifsv M ^Z$, as desired.
\end{proof}

Taking products of parameter spaces hence shows that the infimum in the 
definition of integral foliated simplicial volume is a \emph{minimum}:

\begin{cor}\label{cor:infpar}
  Let $M$ be an oriented closed connected manifold with fundamental
  group~$\Gamma$. Then there exists a standard
  $\Gamma$-space~$(X,\mu)$ with essentially free $\Gamma$-action satisfying
  \[ \ifsv M = \ifsv M ^X.
  \]
\end{cor}
\begin{proof}
  Let $(X_0, \mu_0)$ be a standard $\Gamma$-space with essentially
  free $\Gamma$-action, e.g., the Bernoulli shift of~$\Gamma$
  (Example~\ref{exa:bernoullishift}) (or, in the case of
  finite~$\Gamma$ just $\Gamma$ with the normalised counting measure).
  For $n\in \N_{>0}$ let $(X_n, \mu_n)$ be a standard $\Gamma$-space
  with
  \[ \ifsv M ^{X_n} \leq \ifsv M + \frac 1n. 
  \]
  Then the diagonal $\Gamma$-action on~$(X,\mu) := \bigl(\prod_{n \in
    \N} X_n, \bigotimes_{n \in \N} \mu_n\bigr)$ is essentially free
  and we obtain $\ifsv M ^X = \ifsv M$ from Proposition~\ref{prop:prodpar}.
\end{proof}

%\clcomm{add inverse limits and upper bound; or change subsection title. problem: surjectivity!}

%%%%%%%%%%%%%%%%%%%%%%%%%%%%%%%%%%%%
\subsection{Convex combinations of parameter spaces}

\begin{prop}[convex combinations of parameter spaces]\label{prop:convcombpar}
  Let $M$ be an oriented closed connected $n$-manifold with
  fundamental group~$\Gamma$, let $(X,\mu)$ and $(Y,\nu)$ be two
  standard $\Gamma$-spaces, and let $t \in [0,1]$. Then
  \[   \ifsv M ^Z
     = t \cdot \ifsv M ^X + (1-t) \cdot \ifsv M ^Y, 
  \]
  where $Z := X \sqcup Y$ is the disjoint union of~$X$ and~$Y$ endowed with the
  obvious $\Gamma$-action and the probability measure~$\zeta := t \cdot
  \mu \sqcup (1-t) \cdot \nu$. 
\end{prop}
\begin{proof}
  Under the mutually inverse $\Z \Gamma$-isomorphisms
  \begin{align*}
    \linfz {Z,\zeta} & \longleftrightarrow \linfz {X,\mu} \oplus \linfz {Y,\nu}\\
    f & \longmapsto (f|_X, f|_Y) \\
    \chi_X \cdot f + \chi_Y \cdot g & \longmapsfrom (f,g)
  \end{align*}
  the constant function~$1$ on~$Z$ corresponds to~$(1,1)$, and for 
  all~$f \in \linfz {Z,\zeta}$ we have
  \[ \int_Z f \,d\zeta = t \cdot \int_X f|_X \,d\mu + (1-t) \cdot \int_Y f|_Y \,d\nu. 
  \]
  Therefore, the same arguments as in Proposition~\ref{prop:genprinc}
  show that under the induced mutually inverse chain isomorphisms
  \begin{align*} 
    \linfz{Z,\zeta} \otimes_{\Z \Gamma} C_n(\ucov M,\Z)
     \longleftrightarrow &
     \phantom{\oplus}\ \linfz{X,\mu} \otimes_{\Z \Gamma} C_n(\ucov M,\Z)
     \\& \oplus
     \linfz{Y,\nu} \otimes_{\Z \Gamma} C_n(\ucov M,\Z)
  \end{align*}
  $Z$-parametrised fundamental cycles correspond to pairs of
  $X$-parametrised and $Y$-parametrised fundamental cycles and that 
  (by applying the arguments in both directions)
  \begin{align*}  
     \ifsv M ^Z
     & \geq t \cdot \ifsv M ^X + (1-t) \cdot \ifsv M ^Y, \quad\text{and}\\
     \ifsv M ^Z
     & \leq t \cdot \ifsv M ^X + (1-t) \cdot \ifsv M ^Y. \qedhere 
  \end{align*}
\end{proof}

In combination with Example~\ref{exa:trivialpar} we obtain:

\begin{cor}
  Let $M$ be an oriented closed connected manifold with fundamental
  group~$\Gamma$. Then 
  \[ \bigl\{ \ifsv M ^X \bigm| \text{$(X,\mu)$ is a standard $\Gamma$-space} \bigr\} 
     = \bigl[ \ifsv M, \isv M \bigr] \subset \R.
  \]
\end{cor}

%%%%%%%%%%%%%%%%%%%%%%%%%%%%%%%%%%%
\subsection{Ergodic decomposition of parameter spaces}

We will now show that ergodic parameter spaces suffice to describe 
the integral foliated simplicial volume:

\begin{prop}[ergodic parameters suffice]
  \label{prop:ergpar}
  Let $M$ be an oriented closed connected manifold with fundamental
  group~$\Gamma$. 
  \begin{enumerate}
    \item If $(X,\mu)$ is a standard \mbox{$\Gamma$-space} and $\varepsilon \in \R_{>0}$, 
      then there is a \mbox{$\Gamma$-in}\-variant ergodic probability measure~$\mu'$ on 
      the measurable $\Gamma$-space~$X$ with
      \[ \ifsv M^{(X,\mu')} \leq \ifsv M^{(X,\mu)} + \varepsilon. 
      \]
    \item In particular: For every~$\varepsilon \in \R_{>0}$ there is 
      an ergodic standard $\Gamma$-space~$X$ with
      \[ \ifsv M ^X \leq \ifsv M + \varepsilon. 
      \]
  \end{enumerate}
\end{prop}

The proof of this proposition relies on the ergodic decomposition 
theorem:

\begin{thm}[ergodic decomposition~\protect{\cite[Theorem~5]{farrell}\cite[Theorem~4.2]{varadarajan}}]
  \label{thm:ergdec}
  Let $\Gamma$ be a countable group and let $(X,\mu)$ be a standard
  $\Gamma$-space.  Then there is a probability space~$(P,\nu)$ and a
  familiy~$(\mu_p)_{p\in P}$ of $\Gamma$-ergodic probability measures
  on the measurable $\Gamma$-space~$X$ with the following property: 
  For each Borel subset~$A \subset X$, the function
  \begin{align*}
    P & \longrightarrow [0,1] \\
    p & \longmapsto \mu_p(A)
  \end{align*}
  is measurable and
  \[ \mu(A) = \int_{P} \mu_p(A) \,d\nu(p).
  \]
\end{thm}

In view of this theorem all standard $\Gamma$-spaces can be seen as an
assembly of ergodic $\Gamma$-spaces. However, we have to be careful about 
the sets of measure~$0$ with respect to the involved measures. Therefore, 
we consider the following ``strict'' function spaces and chain complexes:
%\clcomm{already needed earlier?!}

\begin{defi}
  Let $X$ be a measurable space. We write $B(X,\Z)$ for the set of bounded
  measurable functions of type~$X \longrightarrow \Z$. If $\mu$ is a
  measure on~$X$, we write
  \[ N(X,\mu,\Z) := \bigl\{ f \in B(X,\Z)
                    \bigm | \mu \bigl( f^{-1}( \Z \setminus 0) \bigr) = 0
                    \bigr\} \]
  for the set of all functions vanishing $\mu$-almost everywhere. 
\end{defi}

\begin{rem}\label{rem:strictch}
  Let $M$ be an oriented closed connected manifold with fundamental
  group~$\Gamma$ and let $(X,\mu)$ be a standard $\Gamma$-space. By
  definition,
  \[ \linfz{(X,\mu)} \cong B(X,\Z) / N(X,\mu,\Z), \]
  and this isomorphism of $\Z \Gamma$-modules gives rise to an isomorphism 
  \[ \linfz {(X,\mu)} \otimes_{\Z \Gamma} C_*(\ucov M,\Z) 
     \cong 
     \frac{B(X,\Z) \otimes_{\Z \Gamma} C_*({\ucov M},\Z)}%
          {N(X,\mu,\Z) \otimes_{\Z \Gamma} C_*({\ucov M},\Z)}
  \]
  of chain complexes (because the chain modules of~$C_*(\ucov M,\Z)$ are free, and
  hence flat, over~$\Z \Gamma$). 
\end{rem}

\begin{proof}[Proof of Propostion~\ref{prop:ergpar}]
  It suffices to prove the first part. Let $(X,\mu)$ be a standard
  $\Gamma$-space, let $\varepsilon \in \R_{>0}$, and let $n := \dim
  M$. Then there is an $(X,\mu)$- parametrised fundamental cycle~$c = \sum_{j=0}^k
  f_j \otimes \sigma_j \in \linfz{(X,\mu)} \otimes_{\Z \Gamma} C_n ({\ucov
    M},\Z)$ with
  \[ \sum_{j=0}^k \int_X |f_j|\,d\mu \leq \ifsv M^{(X,\mu)} + \varepsilon. \] 
  Let $c_\Z \in \Z \otimes_{\Z \Gamma} C_n ({\ucov M},\Z)$ be an integral 
  fundamental cycle of~$M$. Because $c$ is an $(X,\mu)$-parametrised
  fundamental cycle of~$M$, we can find a chain $d \in \linfz{(X,\mu)}
  \otimes_{\Z \Gamma} C_{n+1}({\ucov M},\Z)$ satisfying
  \[ c - c_\Z = \partial d 
     \in \linfz{(X,\mu)} \otimes_{\Z \Gamma} C_n({\ucov M},\Z). \] 
  In view of Remark~\ref{rem:strictch}, we can assume that the
  coefficients~$f_0, \dots, f_k$ of~$c$ and those of~$d$ lie in the
  ``strict'' function space~$B(X,\Z)$ and that there is a (without
  loss of generality, $\Gamma$-invariant) $\mu$-null set~$A \subset X$
  and a chain~$c' \in B(X,\Z) \otimes_{\Z \Gamma} C_n ({\ucov M}, \Z)$
  satisfying the relation
  \[ c - c_\Z = \partial d + \chi_A \cdot c'
     \in B(X,\Z) \otimes_{\Z \Gamma} C_n ({\ucov M},\Z). \]
  Here, $\chi_A \cdot c'$ uses the canonical $B(X,\Z)^\Gamma$-$\Z\Gamma$-bimodule 
  structure on~$B(X,\Z)$.

  By the ergodic decomposition theorem (Theorem~\ref{thm:ergdec}), we
  obtain a probability space~$(P,\nu)$ and $\Gamma$-invariant
  ergodic probability measures~$(\mu_p)_{p \in P}$ on~$X$ with
  \begin{align*}
     \mu(B) = \int_X \mu_p(B) \,d\nu(p)
  \end{align*}
  for all Borel sets~$B \subset X$. Hence, for all~$f \in B(X,\Z)$ we have
  \[ \int_X f \,d\mu = \int_P \int_X f \,d\mu_p \,d\nu(p).
  \]
  Taking $f := \sum_{j=0}^k |f_j|$ and keeping in mind that $A$ is a
  $\mu$-null set, we thus find a~$p \in P$ with 
  \[ \mu_p(A) = 0 
     \quad\text{and}\quad 
     \int_X f \,d\mu_p \leq \int_X f \,d\mu.
  \]
  
  We now show that $\ifsv M ^{(X, \mu_p)} \leq \int_X f \,d\mu\leq \ifsv M ^{(X,\mu)} +
  \varepsilon$: To this end, we consider the chain
  \[ c_p := [c] \in 
         \frac{B(X,\Z) \otimes_{\Z \Gamma} C_n({\ucov M},\Z)}%
              {N(X, \mu_p,\Z) 
                \otimes_{\Z \Gamma} C_n({\ucov M},\Z)} 
     \cong \linfz{(X,\mu_p)} 
       \otimes_{\Z \Gamma} C_n({\ucov M},\Z).
  \]
  Then $c_p$ is an $(X,\mu_p)$-parametrised fundamental cycle of~$M$, because 
  $c - c_\Z = \partial d + \chi_A \cdot c'$ holds in the ``strict''
  twisted chain complex~$B(X,\Z) \otimes_{\Z \Gamma} C_*({\ucov
  M},\Z)$ and $\mu_p(A) = 0$, and so
  \[ c_p - c_\Z = \partial [d] 
     \in \linfz{(X, \mu_p)}
         \otimes_{\Z \Gamma} C_n({\ucov M},\Z). 
  \]
  Furthermore, we obtain the desired estimate for the norms, namely
  \begin{align*}
            \ifsv {c_p} ^{(X,\mu_p)}
    & \leq     \sum_{j=0}^k \int_X |f_j| \,d\mu_p 
      =     \int_X f \,d\mu_p 
      \leq  \int_X f \,d\mu 
      \leq  \ifsv M^{(X,\mu)} + \varepsilon. \qedhere
  \end{align*}
\end{proof}

As products of ergodic spaces are not necessarily ergodic, it is not 
clear that there is an analogue of Corollary~\ref{cor:infpar} for ergodic 
parameter spaces:

\begin{question}
  Is the integral foliated simplicial volume always given by an ergodic 
  parameter space? Is the integral foliated simplicial volume always given 
  by the Bernoulli shift of the fundamental group? 
\end{question}

%%%%%%%%%%%%%%%%%%%%%%%%%%%%%%%%%%%%
\subsection{Integral foliated simplicial volume and finite coverings}

We will now prove that integral foliated simplicial volume is multiplicative 
with respect to finite coverings:

\begin{thm}[multiplicativity of integral foliated simplicial volume]
  \label{thm:ifsvmult}
  Let $M$ be an oriented closed connected $n$-manifold and let $p
  \colon N \longrightarrow M$ be a $d$-sheeted covering with $d \in
  \N_{>0}$. Then
  \[ \ifsv M = \frac1d \cdot \ifsv N. 
  \]
\end{thm}

The theorem will follow from compatibility with respect to restriction
and induction of parameter spaces (Proposition~\ref{prop:res}
and~\ref{prop:ind}).

\begin{setup}\label{set:fincov}
  Let $M$ be an oriented closed connected $n$-manifold with
  fundamental group~$\Gamma$, and let $p \colon N \longrightarrow M$
  be a $d$-sheeted covering with $d \in \N_{>0}$.  Let $\Lambda$ be
  the fundamental group of~$N$, and let $\Lambda' = \pi_1(p)(\Lambda)
  \cong \Lambda$ be the subgroup of~$\Gamma$ associated with~$p$
  (which has index~$d$ in~$\Gamma$). For notational simplicity, 
  in the following, we will identify the groups~$\Lambda$ 
  and $\Lambda'$ via the isomorphism given by~$p$.
\end{setup}

For the discussion of induction spaces and associated constructions,
it will be necessary to choose representatives:

\begin{setup}\label{set:reps}
  Let $\Gamma$ be a countable group, let $\Lambda \subset \Gamma$ be a
  subgroup of finite index~$d := [\Gamma : \Lambda]$, and let $g_1, \dots, g_d \in \Gamma$ 
  be a set of representatives of~$\Lambda$ in~$\Gamma$:
  \[ \{ g_1 \cdot \Lambda, \dots, g_d \cdot \Lambda\} = \Gamma / \Lambda.
  \]
\end{setup}

\begin{defi}[induction]\label{def:induction}
  In the situation of Setup~\ref{set:reps}, let $(Y,\nu)$ be a
  standard $\Lambda$-space. Then the \emph{induction~$(\Gamma
    \times_\Lambda Y, \mu)$ of~$(Y,\nu)$ from~$\Lambda$ to~$\Gamma$}
  is the standard $\Gamma$-space defined as follows:
  \begin{itemize}
    \item The set
      \[ \Gamma \times_\Lambda Y 
         := \Gamma \times Y \bigm/
            \{ (g \cdot h , y) \sim (g, h\cdot y) \mid g \in \Gamma, h \in \Lambda, y \in Y 
            \}
      \]
      is endowed with the measurable structure induced from the bijection
      \begin{align*}  
        \Gamma \times_\Lambda Y & \longrightarrow \Gamma / \Lambda \times Y
        \\
        [g_j, y]
        & \longmapsto (g_j \cdot \Lambda, y)
      \end{align*}
      (where $\Gamma/\Lambda$ is given the discrete Borel structure). 
      Moreover, the probability measure~$\mu$ is the pull-back of the 
      measure~$1/d \cdot \nu' \otimes \nu$ on~$\Gamma/ \Lambda \times Y$ 
      under this bijection, 
      where $\nu'$ denotes the counting measure on~$\Gamma/\Lambda$.
    \item 
      The $\Gamma$-action on~$\Gamma \times_\Lambda Y$ is defined by 
      \begin{align*}
        \Gamma \times (\Gamma \times_\Lambda Y )
        & \longrightarrow \Gamma \times_\Lambda Y
        \\
        \bigl(g, [g',y]\bigr) & \longmapsto [g \cdot g', y].
      \end{align*}
  \end{itemize}
\end{defi}

Notice that in the above definition $\mu$ is indeed
$\Gamma$-invariant, and that the measurable structure and the
probability measure~$\mu$ on the induction space do \emph{not} depend
on the chosen set of representatives.

\begin{prop}[induction of parameter spaces]\label{prop:ind}
  In the situation of Setup~\ref{set:fincov}, let $(Y,\nu)$ 
  be a standard $\Lambda$-space. Then
  \[ \ifsv  M ^{\Gamma \times_\Lambda Y}
     = \frac 1d \cdot \ifsv N^Y.
  \]
\end{prop}
\begin{proof}
  We choose representatives~$g_1, \dots, g_d$ for
  the index~$d$ subgroup~$\Lambda \subset \Gamma$ as in
  Setup~\ref{set:reps}. Induction of parameter spaces is compatible
  with algebraic induction of modules: We have (well-defined) mutually
  inverse $\Z\Gamma$-isomorphisms
  \begin{align*}
    \varphi \colon \linfz{\Gamma \times_{\Lambda} Y} & \longrightarrow
    \linfz Y \otimes_{\Z \Lambda} \Z \Gamma
    \\
    f &\longmapsto \sum_{j=1}^d \bigl(y \mapsto f([g_j,y])\bigr) \otimes g_j,
    \\
    \psi \colon \linfz Y \otimes_{\Z\Lambda} \Z \Gamma & \longrightarrow
    \linfz{\Gamma \times_{\Lambda} Y}
    \\
    f \otimes g_j & \longmapsto
    \left( [g_k,y] 
       \mapsto \begin{cases}
         f(y) & \text{if $k=j$}\\
         0    & \text{if $k\neq j$}
         \end{cases}
    \right).
  \end{align*}
  Because $p \colon M\longrightarrow N$ is a finite covering, $M$ and
  $N$ share the same universal covering space~$\ucov M = \ucov N$ and
  the $\Lambda$-action on~$\ucov N$ is nothing but the restriction of
  the $\Gamma$-action on~$\ucov M$. Therefore, the above maps induce
  mutually inverse chain complex isomorphisms
  \begin{align*}
    \Phi \colon 
    \linfz {\Gamma \times_{\Lambda} Y} 
      \otimes_{\Z \Gamma} C_*(\ucov M,\Z)
      & \longrightarrow 
      \linfz Y \otimes_{\Z\Lambda} \Z \Gamma \otimes_{\Z \Gamma} C_*(\ucov M,\Z)\\
      & \phantom{\longrightarrow\,}\text{\makebox[0pt][r]{$\cong$}}\ \,
        \linfz Y \otimes_{\Z\Lambda} C_*(\ucov N,\Z)
      \\
      f \otimes c 
      & \longmapsto
      \sum_{j =1}^d f([g_j,\,\cdot\,]) \otimes g_j \cdot c,
      \\
      \Psi 
      \colon \linfz Y \otimes_{\Z \Lambda} C_*(\ucov N,\Z) 
      & \longrightarrow \linfz{\Gamma\times_\Lambda Y} \otimes_{\Z \Gamma} C_*(\ucov M,\Z)
      \\
      f \otimes c &\longmapsto \psi(f \otimes 1) \otimes c. 
  \end{align*}
  
  It is not difficult to see that $\Phi$ and $\Psi$ map $\Gamma
  \times_{\Lambda} Y$-parametrised fundamental cycles of~$M$ to
  $Y$-parametrised fundamental cycles of~$N$, and vice versa: It
  suffices to prove this claim for~$\Phi$. For this, we use the following 
  transfer type argument: Let $c_\Z = \sum_{j=1}^k a_j
  \otimes \sigma_j \in \Z \otimes_{\Z \Gamma} C_n(\ucov M,\Z) \cong
  C_n(M,\Z)$ be a fundamental cycle of~$M$. By construction,
  \[ \Phi \circ i_M^{\Gamma \times_\Lambda Y}(c_\Z) 
     = i_N^Y\biggl( \sum_{j=1}^k a_j \otimes \sum_{\ell=1}^d g_\ell \cdot \sigma_j
           \biggr),
  \]
  and $g_1 \cdot \sigma_j, \dots, g_d \cdot \sigma_j$ are 
  $\pi_N$-lifts of the $d$ different $p$-lifts of~$\pi_M\circ\sigma_j$,
  where $\pi_N \colon \ucov N \longrightarrow N$ and $\pi_M \colon
  \ucov M \longrightarrow M$ denote the universal covering maps. Therefore, 
  $\sum_{j=1}^k a_j \otimes \sum_{\ell=1}^d g_\ell \cdot \sigma_j$ is a 
  fundamental cycle of~$N$, which proves the claim about parametrised 
  fundamental cycles.
  
  By definition of the induction space (Definition~\ref{def:induction}), 
  the $d$~copies of~$Y$ inside~$\Gamma \times_\Lambda Y$ are each given the 
  weight~$1/d$. Therefore, it is not difficult to show that 
  \[ \ifsv{\Phi(c)}^Y \leq d \cdot \ifsv c^{\Gamma\times_\Lambda Y} 
  \]
  holds for all chains~$c \in \linfz{\Gamma\times_\Lambda Y}
  \otimes_{\Z \Gamma} C_*(\ucov M,\Z)$, and that
  \[ \ifsv{\Psi(c)}^{\Gamma\times_\Lambda Y} \leq \frac1d \cdot \ifsv c^Y
  \]
  holds for all chains~$c \in \linfz Y \otimes_{\Z \Lambda} C_*(\ucov
  N,\Z)$. Taking the infimum over all parametrised fundamental cycles
  therefore yields that
  \[ \ifsv M ^{\Gamma\times_\Lambda Y} \leq \frac1d \cdot \ifsv N ^Y 
     \quad \text{and}\quad
     \ifsv N ^Y \leq d \cdot \ifsv M ^{\Gamma\times_\Lambda Y}. 
     \qedhere
  \]
\end{proof}

\begin{cor}[coset spaces as parameter space]\label{cor:cosetpar}
  In the situation of Setup~\ref{set:fincov} we have 
  \[ \ifsv M ^{\Gamma/\Lambda} = \frac1d \cdot \isv N. 
  \]
  Here, we equip the finite set~$\Gamma/\Lambda$ with the left
  $\Gamma$-action given by translation and the normalised counting
  measure. 
\end{cor}
\begin{proof}
  Let $X$ be a standard $\Lambda$-space consisting of a single point. Then
  \[ \Gamma/\Lambda 
     \cong \Gamma \times_\Lambda X
  \]
  (in the category of standard $\Gamma$-spaces). Hence, Proposition~\ref{prop:ind} 
  and Example~\ref{exa:trivialpar} show that
  \[ \ifsv M ^{\Gamma/\Lambda} = \frac 1d \cdot \ifsv N^X = \frac1d \cdot \isv N.
     \qedhere
  \]
\end{proof}

Conversely, we will now consider restriction of parameter spaces:

\begin{defi}[restriction]
  Let $\Gamma$ be a group, let $(X,\mu)$ be a standard $\Gamma$-space, and
  let $\Lambda \subset \Gamma$ be a subgroup. Restricting the
  $\Gamma$-action on~$X$ to~$\Lambda$ (and keeping the same
  probability measure) results in a standard $\Lambda$-space, the
  \emph{restriction~$\res^\Gamma_\Lambda (X,\mu)$ of~$(X,\mu)$ from~$\Gamma$
    to~$\Lambda$}.
\end{defi}

\begin{prop}[restriction of paramater spaces]\label{prop:res}
  In the situation of Setup~\ref{set:fincov} let $(X,\mu)$ 
  be a standard $\Gamma$-space. Then
  \[ \frac1d \cdot \ifsv N ^{\res^\Gamma_\Lambda X}
     \leq 
     \ifsv M^X. 
  \]
\end{prop}
\begin{proof}
  In view of Proposition~\ref{prop:ind}, it suffices to show that
  \[ \ifsv M ^{\Gamma \times_\Lambda \res^\Gamma_\Lambda X} \leq \ifsv M^X. 
  \]
  The map 
  \begin{align*}
    \Gamma \times_\Lambda \res^\Gamma_\Lambda X 
    & \longrightarrow X \\
    [g,x] & \longmapsto g\cdot x
  \end{align*}
  satisfies the hypotheses of Proposition~\ref{prop:genprinc}. Therefore, 
  we obtain the desired estimate $\ifsv M ^{\Gamma \times_\Lambda \res^\Gamma_\Lambda X} \leq \ifsv M^X$.
\end{proof}

\begin{exa}
  In the situation of Proposition~\ref{prop:res}, in general, equality
  will \emph{not} hold. For example, we could consider a double
  covering~$S^1 \longrightarrow S^1$ and a parameter space for the
  base manifold consisting of a single point.
\end{exa}

We will now complete the proof of Theorem~\ref{thm:ifsvmult}:

\begin{proof}[Proof of Theorem~\ref{thm:ifsvmult}]
  From Proposition~\ref{prop:res} we obtain, by taking the infimum
  over all standard $\Gamma$-spaces as parameter spaces for~$M$,
  \[ \frac 1d \cdot \ifsv N \leq \ifsv M. 
  \]
  Conversely, from Proposition~\ref{prop:ind} we obtain, by taking the
  infimum over all standard $\Lambda$-spaces as parameter spaces for~$N$, 
  \[ \ifsv M \leq \frac 1d \cdot \ifsv N. 
  \qedhere
  \]
\end{proof}

%%%%%%%%%%%%%%%%%%%%%%%%%%%%%%%%%%%%%%%%%%%%%%%%%%%%%%%%%%%%%%%%%%%%%%%%%%%%%%%%
\section{A proportionality principle for\\ integral foliated simplicial volume}
\label{sec:pp}

In this section, we will provide a proof of the proportionality
principle Theorem~\ref{thm:ppgen} and of Corollary~\ref{cor:pphyp}.
We will use the language of measure equivalence of groups and
techniques of Bader, Furman, and Sauer~\cite[Theorem 1.9]{bfs}.

%%%%%%%%%%%%%%%%%%%%%
\subsection{Measure equivalence}
The notion of measure equivalence was originally introduced by
Gromov~\cite[0.5.E1]{Gromov2} as a measure-theoretic analogue of
quasi-isometry.
\begin{defi}[measure equivalence]
Two countable groups $\Ga$ and $\La$ are called \emph{measure equivalent} (ME) if there is a standard 
measure space $(\Omega,\mu)$ with commuting measure preserving $\Ga$- and $\La$-actions, 
such that each of the actions admits a finite measure fundamental domain~$X_\Ga$ and $X_\La$ respectively. 
The space $(\Omega,\mu)$ endowed with these actions is called an \emph{ME-coupling} of $\Ga$ and $\La$.
The ratio $c_\Omega=\mu(X_\La)/\mu(X_\Ga)$ is independent of the
chosen fundamental domains and is called the \emph{coupling index}
of the ME-coupling~$\Omega$.
\end{defi}

In our context, the fundamental domains do not need to be strict
fundamental domains; it suffices that $\Omega = \bigcup_{\gamma \in
  \Gamma} \gamma \cdot X_\Gamma$ is a disjoint decomposition up to
measure~$0$. The cocycle in Definition~\ref{cocycle} then will only be
well-defined up to sets of measure~$0$; however, this poses no
problems in the sequel as we will pass to $L^1$- and $L^\infty$-spaces
anyway.

\begin{exa}[lattices]\label{ex:our case}
A second countable locally compact group $G$ with its Haar measure is
an ME-coupling for every pair of lattices $\Ga$ and $\La$ in
$G$. Indeed, the Haar measure on~$G$ is bi-invariant (because $G$
contains lattices), and the left actions
\begin{align*}
  \Gamma \times G & \longrightarrow G
  &
  \Lambda \times G & \longrightarrow G
 \\
  (\gamma, g) & \longmapsto \gamma g
 &
  (\lambda, g) & \longmapsto g \lambda^{-1}
\end{align*}
of~$\Gamma$ and $\Lambda$ on~$G$ given by multiplication in~$G$
commute with each other.
\end{exa}

\begin{setup}\label{ME}
Let $(\Omega,\mu)$ be an ME-coupling of $\Ga$ and $\La$. We suppose that both 
the actions are left actions, and
we fix a fundamental domain for each of the
two actions on the ME-coupling, denoted as $X_\Ga$ and $X_\La$ respectively.
\end{setup}
\begin{defi}[ergodic/mixing ME-coupling]
In the situation of Setup~\ref{ME}, the ME-coupling is \emph{ergodic} 
(resp.~\emph{mixing}) if the $\Ga$-action on $\La \backslash\Omega$ is 
ergodic (resp.~mixing) and the $\La$-action on
$\Ga\backslash\Omega$ is ergodic (resp.~mixing).
\end{defi}
\begin{rem}\label{ME_ergodicity}
Note that in this situation
the $\Ga$-action on $\La\backslash\Omega$ is ergodic if and only if the $\La$-action on $\Ga\backslash\Omega$ is ergodic \cite[Lemma 2.2]{Furman}.
\end{rem}

\begin{defi}[ME-cocycle]\label{cocycle}
In the situation of the Setup~\ref{ME}, we define the 
measurable cocycle $\alpha_\La$ associated to $X_\La$ as the map
$$\alpha_\La:\Ga\times X_\La\rightarrow \La$$
such that $\alpha_\La(\gamma,x)$ is the unique element satisfying
$\gamma x\in \alpha_\La(\gamma,x)^{-1}X_\La$
for all $x \in X_\La$ and $\gamma \in \Ga$. Similarly, we define 
$\alpha_\Ga:\La\times X_\Ga\rightarrow \Ga$.
If we choose another fundamental domain for $\La\backslash \Omega$
then the associated cocycle is measurably cohomologous to
$\alpha_\La$ \cite[Section 2]{Furman} (the same for $\alpha_\Ga$).
\end{defi}
With this notation the natural left action of $\Ga$ on 
$X_\La$ and of $\La$ on $X_\Ga$
is described as follows:
\begin{equation}\label{eq:action}
\begin{aligned}
    \Gamma \times X_\Lambda & \longrightarrow X_\Lambda
  & \Lambda \times X_\Gamma & \longrightarrow X_\Gamma 
    \\
    (\gamma,x) & \longmapsto \gamma \bullet x := \alpha_\Lambda(\gamma,x) \gamma x
  & (\lambda,y) &\longmapsto \lambda \bullet y:= \alpha_\Gamma(\lambda,y)\lambda y,
\end{aligned}
\end{equation}
where we write $\gamma \bullet x$ to distinguish
it from the action $\gamma x$ of $\Gamma$ on~$\Omega$.
\begin{rem}\label{rem:isom_action}
In the situation of the Setup~\ref{ME},
consider the standard $\Gamma$-space 
$(X_\La,\mu_\La=\mu(X_\La)^{-1}{\mu_|}_{X_\La})$ 
with $\Ga$-action described above and 
the standard $\Ga$-space $\La\backslash\Omega$
with the probability measure induced from $\mu$
and the left translation $\Ga$-action.
Then the map $X_\La\hookrightarrow \Omega \rightarrow \La\backslash\Omega$  
is a measure isomorphism.
\end{rem}

\begin{defi}[bounded ME-coupling]
In the situation of the Setup~\ref{ME},
assume that $\La$ is finitely generated, and let $l:\La\rightarrow \N$ 
be the length function associated to some word-metric on $\La$. 
We say that the fundamental domain~$X_\La$ is \emph{bounded} if, 
for every $\gamma \in \Ga$, the function $x \mapsto l(\alpha_\La(\gamma,x))$ 
is in~$L^\infty(X_\La,\R)$. 

Let $\Ga$ and $\La$ be finitely generated. An ME-coupling of $\Ga$ and 
$\La$ is \emph{bounded}
if it admits bounded $\Ga$- and $\La$-fundamental domains.
\end{defi}
\begin{exa}\label{rem:our case}
A connected second countable locally compact group $G$ with 
its Haar measure is a bounded ME-coupling 
for every pair of uniform lattices in~$G$ \cite[p. 321]{bfs2} 
\cite[Corollary 6.12 p. 58]{stro}. 
\end{exa}

%%%%%%%%%%%%%%%%%%%%%%%%
\subsection{Homology of groups}
In the aspherical case, we can express integral foliated simplicial volume in terms of group homology:
Let $\Gamma$ be a discrete group. The \emph{bar resolution} of $\Gamma$ 
is the $\Z\Ga$-chain complex $C_\ast(\Ga)$ defined as follows: 
for each~$n\in \N$ let
$$C_n(\Ga)=\biggl\{ \sum_{\gamma\in \Ga^{n+1} }
a_\ga\cdot\gamma_0\cdot[\gamma_1|\cdots|\gamma_n] 
\biggm| \fa{\ga=(\gamma_0,\dots,\gamma_n)\in \Ga^{n+1}} a_\ga\in\Z\biggr\}$$ 
with the $\Ga$-action characterized by
$$\overline{\gamma} \cdot(\gamma_0\cdot[\gamma_1|\cdots|\gamma_n])=
(\overline{\gamma}\cdot\gamma_0)\cdot[\gamma_1|\cdots|\gamma_n] $$
for all $\overline{\gamma}\in \Ga$ and all $\gamma\in \Ga^{n+1}$.
The differential $\partial_\ast:C_\ast(\Ga)\rightarrow C_{\ast-1}(\Ga)$ is
defined by
$$
\begin{array}{rcl}
C_n(\Ga)&\longrightarrow &C_{n-1}(\Ga)\\
\gamma_0\cdot[\gamma_1|\cdots|\gamma_n]& \longmapsto & \gamma_0\cdot\gamma_1\cdot 
[\gamma_2|\cdots|\gamma_n]\\
& &+\sum_{j=1}^{n-1} (-1)^j \cdot \gamma_0\cdot[\gamma_1|\cdots|\gamma_{j-1}|\gamma_j
\cdot \gamma_{j+1}|\gamma_{j+2}|\cdots|\gamma_n]\\
& & +(-1)^n \cdot \gamma_0\cdot[\gamma_1|\cdots|\gamma_{n-1}].
\end{array}
$$ Moreover, the bar resolution is a normed chain complex (i.e., the
differentials in each degree are bounded operators) with the
$\ell^1$-norm given by
$$\lone[bigg]{ \sum_{\gamma\in \Ga^{n+1} }
a_\ga\cdot\gamma_0\cdot[\gamma_1|\cdots|\gamma_n]}= \sum_{\gamma\in \Ga^{n+1} }
|a_\ga|.$$

We obtain a version of the bar resolution with coefficients using 
the tensor product: For every normed right $\Z\Ga$-module $A$ let
$$C_\ast(\Ga,A):=A\otimes_{\Z\Ga} C_\ast(\Ga).$$ 
\begin{defi}[group homology]
Let $\Ga$ be a discrete group and $A$ be a normed right $\Z\Ga$-module.
Then the \emph{group homology of $\Ga$ with coefficients in $A$} is
$$H_\ast(\Ga,A):=H_\ast(C_\ast(\Ga,A)).$$
The $\ell^1$-norm on the chain complex induces an $\ell^1$-semi-norm on 
the group homology.
\end{defi}

\begin{prop}\label{prop:asp}
Let $M$ be an aspherical manifold with universal covering~$\widetilde{M}$,
 and fundamental group $\Ga$. Let $A$ be a normed right 
$\Z\Ga$-module. Then there exists a natural chain map
$$c_\Ga:C_\ast(\widetilde{M},A)\longrightarrow C_\ast(\Ga,A)$$
that induces an isometric isomorphism
$$c_\Ga:H_\ast(M,A)\longrightarrow H_\ast(\Ga,A).$$
\end{prop}
\begin{proof}
  It is not difficult to see that the classical
  mutually inverse $\Z \Gamma$-chain homotopy equivalences~$C_*(\ucov M,\Z)
  \longleftrightarrow C_*(\Gamma,\Z)$, defined using a $\Gamma$-funda\-men\-tal domain
  on~$\ucov M$, are norm non-increasing. This gives the desired isometric 
  isomorphism in homology. 
\end{proof}

\begin{cor}\label{cor:homology}
Let $M$, $\Ga$, and $(X,\mu)$ be as in Definition~\ref{def:pfc}. If
$M$ is aspherical, there exists an isometric isomorphism
$$H_\ast(M,\linfz {X})\stackrel{\cong}{\longrightarrow}
H_\ast(\Ga,\linfz {X}),$$
with respect to the $\ell^1$-semi-norm induced by the $\ell^1$-norm on 
$\linfz{X}$.
\end{cor}

\begin{lem}\label{lem:dense}
Let $\Gamma$ be a countable group and 
let $(X,\mu)$ be a standard $\Ga$-space as in Definition~\ref{def:pfc}.
The inclusion 
$
\linfz {X}\hookrightarrow L^1(X,\Z) 
$
has dense image with respect to the $\ell^1$-norm and induces an isometric map
$$H_\ast(\Ga,\linfz {X})\longrightarrow H_\ast(\Ga,L^1(X,\Z)).$$
\end{lem}
\begin{proof}
  Inclusions of dense subcomplexes induce isometric maps on
  homology~\cite[Lemma~2.9]{mschmidt}\cite[Proposition~1.7]{loehphd}
  (the cited proofs also carry over to this integral setting).
\end{proof}

\begin{rem}\label{rem:l1}
Let $M$, $\Ga$, and $(X,\mu)$ be as in Definition~\ref{def:pfc}.
If $M$ is aspherical, by Corollary~\ref{cor:homology} and Lemma~\ref{lem:dense}
we deduce that the $X$-parametrised simplicial volume of $M$ can be computed
via the $\ell^1$-semi-norm on the group homology $H_\ast(\Ga,L^1(X,\Z))$.
\end{rem}

%%%%%%%%%%%%%%%%%%%%%%%%%%%%%%%%%%%%
\subsection{Proportionality principle, general case}
\begin{proof}[Proof of Theorem~\ref{thm:ppgen}]
Let $n\in \N$. Let  $\Ga$ and $\La$ be
fundamental groups of oriented closed connected aspherical 
$n$-manifolds $M$ and $N$ with positive simplicial volume.
Assume that
$(\Omega,\mu)$ is an ergodic bounded ME-coupling of $\Ga$ and $\La$.
Fix  a bounded fundamental domain $X_\Ga\subset \Omega$ 
(resp.~$X_\La\subset \Omega$) of the  $\Ga$-action on 
$\Omega$ (resp.~of the $\La$-action on $\Omega$) and 
let $\alpha_\Ga$ (resp.~$\alpha_\La$) be the associated cocycle (Definition
\ref{cocycle}). 

Using the cocycles associated with the ME-coupling we translate 
para\-metrised fundamental cycles of one manifold into
parametrised fundamental cycles of the other manifold, while
controlling the $\ell^1$-semi-norm.
More precisely:
consider the standard space $(X_\Ga,\mu_\Ga=\mu(X_\Ga)^{-1}{\mu_|}_{X_\Ga})$ 
with the $\La$-action defined in Equation~\eqref{eq:action}.
By Remark~\ref{rem:l1} we use the $\ell^1$-norm  on 
$L^1(X_\Ga,\Z)\otimes_{\Z \La}C_\ast(\La,\Z)$ to estimate the $X_\Ga$-parametrised 
simplicial volume of $N$.
Let  $\La^{\ast+1}\times X_\Ga$ be endowed with
the diagonal $\La$-action and which carries the product
of the counting measure and $\mu_\Ga$. We identify
$L^1(X_\Ga,\Z)\otimes_{\Z \La}C_\ast(\La,\Z)$
with $ L^1(\La^{\ast+1}\times X_\Ga,\Z)^{\text{fin}}_\La,$ where 
the superscript ``fin'' 
indicates the submodule of functions $f: \La^{*+1}\times X_\Ga\rightarrow \Z$ 
with the property that there is a finite subset $F$ of $\La^{*+1}$ such that 
$f$ is supported on $F\times X_\Ga$, and where the subscript $\La$ 
indicates the co-invariants.

Now, let us consider the measurable, countable-to-one, locally measure
preserving (up to a constant factor) map~\cite[p.~284]{bfs}
$$
\begin{array}{rccl}
\phi_n^{(\alpha_\Ga)}: & \La^{n+1}\times X_\Ga& \longrightarrow & \Ga^{n+1}\times X_\La\\
& (\la_0,\dots, \la_n,y)& \longmapsto & \bigl(\alpha_\Ga(\la_0^{-1},y)^{-1},\dots,\alpha_\Ga(\la_n^{-1},y)^{-1}, \La y\cap X_\La\bigr)
\end{array}
$$
and define
$$
\begin{array}{rccl}
 {(\alpha_\Ga)}^\Z_n:& L^1(\La^{n+1}\times X_\Ga,\Z)^{\text{fin}}_\La& \longrightarrow & L^1(\Ga^{n+1}\times X_\La,\Z)^{\text{fin}}_\Ga\\
&f &\longmapsto & (\gamma,x) \mapsto {(\alpha_\Ga)}^\Z_n(f)({\ga},x):=\\
&&&\sum_{({\la},y) \in (\phi_n^{(\alpha_\Ga)})^{-1} ({\ga},x)} f({\la},y).
\end{array}
$$
%where we abbreviated $\overline{\la}=(\la_0,\dots,\la_n)$ and $\overline{\ga}=(\ga_0,\dots, \ga_n)$.
This map is the restriction to integral chains of the map described by
Bader, Furman and Sauer \cite[Theorem 5.7]{bfs}\footnote{In
  Definition~\ref{def:pfc} we define a right action on $\linfz{X}$
  (and similarly on $L^1(X,\Z)$) unlike \emph{op.~cit.}, in which
  $\linfz{X}$ is endowed with a left action.}  up to rescaling by the
coupling index $c_\Omega=\mu(X_\La)/\mu(X_\Ga)$.  The sum on the right
hand side is a.e.~finite~\cite[Lemma~5.8]{bfs}, and the boundedness of
the coupling guarantees that ${(\alpha_\Ga)}^\Z_n(f)$ has finite
support whenever $f$ has finite support; moreover, a straightforward 
computation shows that the map is well-defined on the level of co-invariants. Hence, for a bounded ME-coupling of
$\Gamma$ and $\Lambda$, ${(\alpha_\Ga)}^\Z_n$ is a well-defined norm
non-increasing (up to rescaling by the coupling index) chain map, and
${(\alpha_\Ga)}^\Z_n$ induces a map
$$H\bigl({(\alpha_\Ga)}^\Z_n\bigr):H_n\bigl(\La,L^1(X_\Ga,\Z)\bigr)\longrightarrow H_n\bigl(\Ga,L^1(X_\La,\Z)\bigr)$$
of norm at most~$1/c_\Omega$.

Let us now consider the diagram in Figure~\ref{fig:homology}.
\begin{figure}
    \begin{center}
      \begin{tikzpicture}
        \def\x{4.5}
        \def\y{-1.2}
        \node (A0_0) at (0*\x, 0*\y) {$H^{\ell^1}_n(\La,L^1(X_\Ga,\R))$};
        \node (A1_0) at (1*\x, 0*\y) {$H^{\ell^1}_n(\Ga,L^1(X_\La,\R))$};
        \node (A0_1) at (0*\x, 1*\y) {$H_n(\La,L^1(X_\Ga,\R))$};
        \node (A1_1) at (1*\x, 1*\y) {$H_n(\Ga,L^1(X_\La,\R))$};
        \node (A0_2) at (0*\x, 2*\y) {$H_n(\La,L^1(X_\Ga,\Z))$};
        \node (A1_2) at (1*\x, 2*\y) {$H_n(\Ga,L^1(X_\La,\Z))$};
        \node (A0_3) at (0*\x, 3*\y) {$H_n(\La,\Z)$};
        \node (A0_4) at (0*\x, 4*\y) {$H_n(N,\Z)$};
        \node (A1_3) at (1*\x, 3*\y) {$H_n(\Ga,\Z)$};
        \node (A1_4) at (1*\x, 4*\y) {$H_n(M,\Z)$}; 
        \path (A0_0) edge [->] node [auto] {$\scriptstyle{H^{\ell^1}({(\alpha_\Ga)}^\R_n)}$} (A1_0);
        \path (A0_1) edge [->] node [auto] {$\scriptstyle{H({(\alpha_\Ga)}^\R_n)}$} (A1_1);
        \path (A0_2) edge [->] node [auto] {$\scriptstyle{H({(\alpha_\Ga)}^\Z_n)}$} (A1_2);
        \path (A0_1) edge [->] node [auto] {$\scriptstyle{i_\La}$} (A0_0);
        \path (A1_1) edge [->] node [auto] {$\scriptstyle{i_\Ga}$} (A1_0);
        \path (A0_2) edge [->] node [auto] {$\scriptstyle{j^\R_\La}$} (A0_1);
        \path (A1_2) edge [->] node [auto] {$\scriptstyle{j^\R_\Ga}$} (A1_1);
        \path (A0_3) edge [->] node [auto] {$\scriptstyle{j_\La}$} (A0_2);
        \path (A0_4) edge [->] node [auto] {$\scriptstyle{c_\La}$} (A0_3);
        \path (A1_3) edge [->] node [auto] {$\scriptstyle{j_\Ga}$} (A1_2);
        \path (A1_4) edge [->] node [auto] {$\scriptstyle{c_\Ga}$} (A1_3);
      \end{tikzpicture}
    \end{center}
    \caption{Effect of the cocycle in homology \label{fig:homology}}
\end{figure}
Since $M$ and $N$ are aspherical manifolds, Proposition~\ref{prop:asp}
ensures that there exist isometric isomorphisms $c_\La$ and $c_\Ga$. 
The maps $j_\La$, $j_\Ga$, $j^\R_\La$, $j^\R_\Ga$ are the usual change of
coefficients homomorphisms. In particular, by ergodicity 
of both actions, the maps $j_\La$ and $j_\Ga$ are (semi)-norm
non-increasing isomorphisms.
The inclusion of singular chains into $\ell^1$-chains induces the isometric 
homomorphisms $i_\La$ and $i_\Ga$.
Finally $c_\Omega \cdot H^{\ell^1}({(\alpha_\Ga)}^\R_n)$ is the isometric isomorphism of
Bader, Furman and Sauer~\cite[Theorem 5.7]{bfs}
defined between $\ell^1$-homology groups.

By  the diagram in Figure~\ref{fig:homology} we have that 
\begin{equation}\label{fund_class}
H({(\alpha_\Ga)}^\R_n)\circ j^\R_\La\circ j_\La\circ c_\La([N]_\Z)=m \cdot
(j_\Ga^\R\circ j_\Ga\circ c_\Ga)([M]_\Z) 
\end{equation}
for some $m\in \Z$; here, we use that the $\Gamma$-action
on~$X_\Lambda$ is ergodic, and so $H_n(\Gamma,L^1(X_\Lambda,\Z)) \cong
\Z$.  By~$\sv{N}>0$ we have $\lone{i_\La \circ j_\La^\R
  \circ j_\La\circ c_\La([N]_\Z)}>0$. Since
$c_\Omega \cdot H^{\ell^1}({(\alpha_\Ga)}^\R_n)$ is an isometry we deduce that
$$H^{\ell^1}({(\alpha_\Ga)}^\R_n)\circ i_\La \circ j_\La^\R \circ j_\La\circ c_\La([N]_\Z)\neq 0. $$
By the commutativity of the top square in Figure~\ref{fig:homology},  
we have $|m|\geq 1$ in Equation $\eqref{fund_class}$, and hence
\begin{align*}
\lone{(j_\Ga^\R\circ j_\Ga\circ c_\Ga)([M]_\Z)}&\leq 
|m|\cdot \lone{(j_\Ga^\R\circ j_\Ga\circ c_\Ga)([M]_\Z)}\\
&=\lone{H({(\alpha_\Ga)}^\R_n)\circ j^\R_\La\circ j_\La\circ c_\La([N]_\Z)}\\
&\leq
c_\Omega^{-1}\cdot \lone{j^\R_\La\circ j_\La\circ c_\La([N]_\Z)}.
\end{align*}
By interchanging the roles of $\Ga$ and $\La$, we obtain the converse inequality
and it turns out that $|m|=1$.

By the commutativity of the diagram in Figure~\ref{fig:homology} it follows:
$$\ifsv{M}^{X_\La}=\bigl\|j_\Ga\circ c_\Ga([M]_\Z)\bigr\|_1=\bigl\|H({(\alpha_\Ga)}^\Z_n)\circ j_\La\circ c_\La([N]_\Z)\bigr\|_1
\leq c_\Omega^{-1}\cdot\ifsv{N}^{X_\Ga},$$
which implies
$$\ifsv{ N}^{X_\Ga}\geq c_\Omega\cdot\ifsv{M}^{X_\La}.$$
Similarly we have the other inequality.

Since the map of Remark~\ref{rem:isom_action}
%$X_\Ga\hookrightarrow \Omega \rightarrow \Ga\backslash\Omega$  
%is a measure isomorphism and  
induces an isometric isomorphism 
between $H_n(\Ga,L^1(X_\La,\Z))$ and
$H_n(\Ga,L^1(\La\backslash\Omega,\Z))$ (similarly for $X_\Ga$),
the previous equality does not depend on the
choice of the fundamental domains:
$$\ifsv{ N}^{\Ga\backslash\Omega}= c_\Omega\cdot\ifsv{M}^{\La\backslash\Omega}.$$ 

Let us now suppose that $(\Omega,\mu)$ is a \emph{mixing} bounded
ME-coupling of $\Ga$ and $\La$.  Let $(X, \mu_X)$ be an ergodic
standard $\Ga$-space, with the structure of $\La$-space with the
trivial action. Then $(X\times\Omega,\mu_X\otimes\mu)$ is an ergodic
bounded ME-coupling of $\Ga$ and $\La$ with respect to the left
diagonal $\Ga$- and $\La$-actions.  Indeed if $F_\Ga$ (resp.~$F_\La$)
is a bounded fundamental domain for the action of $\Ga$ on $\Omega$
(resp.~of $\La$) then $X_\Ga=X\times F_\Ga$ (resp.~$X_\La=X\times
F_\La$) is a finite measure fundamental domain of the $\Ga$-action on
$X\times\Omega$ (resp.~of the $\La$-action). We can easily trace back
the boundedness of $X_\Ga$ (resp.~$X_\La$) to the one of $F_\Ga$
(resp.~$F_\La$).  Finally, by the property of mixing actions the
$\Ga$-action on $\La\backslash(X\times \Omega)\cong X\times
\La\backslash\Omega$ is ergodic and $(X\times\Omega,\mu_X\otimes\mu)$
is an ergodic ME-coupling (see Definition~\ref{def:erg/mix} and
Remark~\ref{ME_ergodicity}).

Applying now the previous construction to
$(X\times\Omega,\mu_X\otimes\mu)$, we have
$$
\ifsv{N}^{\Ga\backslash(X\times \Omega)}=
c_{X\times\Omega}\cdot\ifsv{ M}^{\La\backslash(X \times \Omega)},
$$
where $c_{X\times\Omega}=(\mu_X\otimes \mu)(X\times X_\La)/(\mu_X\otimes \mu)
(X\times X_\Ga)=\mu(X_\La)/\mu(X_\Ga)=c_\Omega$.
Using Proposition~\ref{prop:prodpar} it follows that
$$
\ifsv{N}\leq \ifsv{N}^{\Ga\backslash(X\times \Omega)}=
c_{\Omega}\cdot\ifsv{ M}^{X \times (\La\backslash\Omega)}\leq c_{\Omega}\cdot\ifsv{M}^X,
$$
and by Proposition~\ref{prop:ergpar} it turns out that
$$
\ifsv{N}\leq c_\Omega\cdot\ifsv{M}.
$$
By interchanging the roles of $\Ga$ and $\La$ we obtain the other inequality.
\end{proof}

%%%%%%%%%%%%%%%%%%%%%%%%%
\subsection{Proportionality principle, hyperbolic case}
For the proof of Theorem~\ref{mainthm} we need a proportionality
principle for hyperbolic manifolds (Corollary~\ref{cor:pphyp}).

\begin{proof}[Proof of Corollary~\ref{cor:pphyp}]
Let $n\in \N$, let $\Gamma$ and $\Lambda$ be uniform lattices in $G :=
\Isom^+(\Hyp^n)$, and $\Hyp^n/\Ga$ and $\Hyp^n/\La$ be the associated
oriented closed connected hyperbolic $n$-manifolds with fundamental
groups $\Ga$ and $\La$.  The group $G$ with its Haar measure $\mu$ is
a bounded ME-coupling with respect to $\Ga$ and $\La$ (see
Examples~\ref{ex:our case} and \ref{rem:our case}).  The ergodicity of
the coupling follows from the Moore Ergodicity Theorem~\cite[Theorem
  III.2.1]{Bekka-Mayer}.  Clearly, this coupling has coupling
index~$\covol(\Lambda)/\covol(\Gamma)$.

Notice that the standard $\Gamma$-space  $\Lambda \backslash G$ 
(where $\Lambda$ acts on~$G$
as in Example~\ref{ex:our case}) is isomorphic to the coset space $G/\La$
with the probability measure induced 
from the Haar measure~$\mu$ and
the left translation $\Gamma$-action; similarly,
the standard $\La$-space $\Ga\backslash G$ with the
$\La$-action induced from Example~\ref{ex:our case} is isomorphic 
to the coset space~$G/\Gamma$ with the left translation $\La$-action.

Hence, Theorem~\ref{thm:ppgen} applied to this ergodic ME-coupling implies:
$$
\ifsv{\Hyp^n/\La}^{G/\Gamma}= \frac{\covol(\La)}{\covol(\Ga)}\cdot\ifsv{\Hyp^n/\Ga}^{G/\Lambda}.
$$

Since the Moore Ergodicity Theorem~\cite[Theorem III.2.1]{Bekka-Mayer}
ensures also that $(G,\mu)$ is a mixing ME-coupling of $\Ga$ and $\La$,
by the second part of Theorem~\ref{thm:ppgen} it turns out that
\[
\ifsv{\Hyp^n/\La}= \frac{\covol(\La)}{\covol(\Ga)}\cdot\ifsv{\Hyp^n/\Ga}.\qedhere
\]
\end{proof}

%%%%%%%%%%%%%%%%%%%%%%%%%%%%%%%%%%%%%%%%%%%%%%%%%%%%%%%%%%%%%%%%%%%%%%%%
\section{Comparing integral foliated simplicial volume and\\ stable integral simplicial volume}
\label{sec:ifsvstisv}

We will now compare integral foliated simplicial volume with stable
integral simplicial volume. We start with a general sandwich estimate
(Proposition~\ref{prop:compsv}), and we will then give more refined
estimates for the hyperbolic case and for special parameter spaces. In
particular, we will prove Corollary~\ref{cor:ifsvstisvhyp} and
Theorem~\ref{thm:ifsvstisvgen}.

%%%%%%%%%%%%%%%%%%%%%%%%%%%%%%%%%%%%
\subsection{Sandwich estimate for integral foliated simplicial volume}

Integral foliated simplicial volume interpolates between ordinary
simplicial volume and stable integral simplicial volume:

\begin{prop}[comparing simplicial volumes]\label{prop:compsv}
  Let $M$ be an oriented closed connected manifold. Then
  \[ \sv M \leq \ifsv M \leq \stisv M. 
  \]
\end{prop}
\begin{proof}
  The first inequality is contained in
  Proposition~\ref{prop:compisv}. The second inequality follows from
  Corollary~\ref{cor:cosetpar} (alternatively: from
  Proposition~\ref{prop:compisv} and Theorem~\ref{thm:ifsvmult}).
\end{proof}

\begin{exa}[integral foliated simplicial volume of surfaces]
  Let $M$ be an oriented closed connected surface of genus~$g(M)$. Then 
  \[ \ifsv M = 
     \begin{cases}
       2 & \text{if $M \cong S^2$}\\
       \sv M = \stisv M = 4 \cdot g(M) - 4& \text{otherwise}.
     \end{cases}
  \]
 Indeed, because $S^2$ is simply connected, $\ifsv{S^2} = \isv {S^2} =
 2$. If $M \not\cong S^2$, then the classical computation of
 simplicial volume of aspherical surfaces~\cite[Section~0.2]{Gromov} shows that $\sv
 M = \stisv M = 4 \cdot g(M) - 4$, and so the claim follows from
 Proposition~\ref{prop:compsv}. 
\end{exa}

Hence, in the case of finite fundamental group, we have:

\begin{cor}[finite fundamental group]
  \label{cor:finitepi1}
  Let $M$ be an oriented closed connected manifold with finite fundamental group. Then 
  \[ \ifsv M  = \frac1{|\pi_1(M)|} \cdot \isv[norm]{\ucov M} = \stisv M, 
  \]
  where $\ucov M$ is the universal covering of~$M$. In contrast, $\sv M = 0$.
\end{cor}
\begin{proof}
  The universal covering~$\ucov M \longrightarrow M$ has
  $|\pi_1(M)|$~sheets; because $\pi_1(M)$ is finite, $\ucov M$ indeed
  is an oriented closed connected manifold. Hence, Theorem~\ref{thm:ifsvmult} gives 
  us
  \[ \ifsv M = \frac1{|\pi_1(M)|} \cdot \ifsv {\ucov M}. 
  \]
  Moreover, $\ifsv {\ucov M} = \isv[norm] {\ucov M}$ because $\ucov M$
  is simply connected (Example~\ref{exa:trivialpar}). So, in
  combination with Proposition~\ref{prop:compsv}, we obtain
  \[ \stisv M \leq \frac1{|\pi_1(M)|} \cdot \isv[norm]{\ucov M} 
     = \ifsv M \leq \stisv M. 
  \]
  On the other hand, simplicial volume of manifolds with finite (more
  generally, amenable) fundamental group is zero~\cite{Gromov}, and
  so~$\sv M = 0$.
\end{proof}

%%%%%%%%%%%%%%%%%%%%%%%%%%%%%%%%%%%%
\subsection{Proof of  Corollary~\ref{cor:ifsvstisvhyp}}

We will now prove Corollary~\ref{cor:ifsvstisvhyp} with help of the
proportionality principle for hyperbolic manifolds
(Corollary~\ref{cor:pphyp}). More precisely, we will prove the following, 
slightly more general, statement:

\begin{thm}\label{thm:ifsvstisvhyppar}
  Let $n\in \N$, and let $M$ and $N$ be oriented closed connected
  hyperbolic $n$-manifolds with fundamental groups~$\Gamma$ and
  $\Lambda$, and let $G := \Isom^+ (\Hyp^n)$. Let $S$ be a set of
  representatives of uniform lattices in~$G$, containing a
  representative for every isometry class of oriented closed connected
  hyperbolic $n$-manifolds. The product~$\prod_{\Lambda' \in S}
  G/\Lambda'$ is a standard $\Gamma$-space with respect to the diagonal
  translation action and the product of the probability measures
  induced by the (bi-invariant) Haar measure on~$G$.  Then
  \[ \ifsv M \leq \ifsv M ^{\prod_{\Lambda' \in S} G/\Lambda'} \leq \frac{\vol(M)}{\vol(N)} \cdot \stisv N. 
  \]
\end{thm}

Notice that in every dimension up to isometry there are only countably
many different oriented closed connected hyperbolic manifolds. Hence,
$S$ in the previous theorem is countable, and so the
product~$\prod_{\Lambda' \in S} G/\Lambda'$ indeed is a standard
$\Gamma$-space.

\begin{proof}
  We can view $\Gamma$ and $\Lambda$ as uniform
  lattices in~$G =\Isom^+(\Hyp^n)$ and we have $M = \Hyp^n/\Gamma$ and
  $\covol(\Gamma) = \vol(M)$ (and similarly
  for~$N$ and its finite coverings). 

  Let $N' \longrightarrow N$ be a finite covering of~$N$, let
  $\Lambda'$ be the fundamental group of~$N'$, and let $d := [\Lambda
    : \Lambda']$ be the number of sheets of this covering. From the
  proportionality principle (Corollary~\ref{cor:pphyp}),
  Proposition~\ref{prop:prodpar}, and Proposition~\ref{prop:compisv}
  we obtain
  \begin{align*}
    \ifsv M & = \ifsv {\Hyp^n/\Gamma} 
             \leq \ifsv{\Hyp^n/\Gamma}^{\prod_{\Lambda''\in S} G/\Lambda''}
             \leq \ifsv{\Hyp^n/\Gamma}^{G/\Lambda'}\\ 
            & = \frac{\covol(\Gamma)}{\covol(\Lambda')} \cdot \ifsv{\Hyp^n/\Lambda'}^{G/\Gamma}
              = \frac{\covol(\Gamma)}{\covol(\Lambda)} \cdot \frac1{[\Lambda:\Lambda']}   
              \cdot \ifsv{N'}^{G/\Gamma}\\
            & \leq \frac{\vol(M)}{\vol(N)} \cdot \frac1d 
              \cdot \isv{N'}.
  \end{align*}
  Taking the infimum over all finite coverings of~$N$ finishes the proof.
\end{proof}

This concludes the proof of Corollary~\ref{cor:ifsvstisvhyp}.

%%%%%%%%%%%%%%%%%%%%%%%%%%%%%%
\subsection{Proof of Theorem{~\ref{thm:ifsvstisvgen}}}

\begin{defi}\label{def:prodpar}
  Let $\Gamma$ be a finitely generated group and let $S$ be a set of finite index
  subgroups of~$\Gamma$ (hence, $S$ is countable). Then the standard $\Gamma$-space~$X_{\Gamma, S}$ 
  is given by the set
  \[ \prod_{\Lambda \in S} \Gamma/\Lambda,
  \]
  equipped with the product measure of the normalised counting measures 
  and the diagonal left $\Gamma$-action given by translating cosets.
\end{defi}

Notice that in the situation of Definition~\ref{def:prodpar} the
$\Gamma$-action on~$X_{\Gamma,S}$ is free if and only if
$\bigcap_{\Lambda \in S} \Lambda = \{e\}$, and that the $\Gamma$-action 
is not ergodic in general.

\begin{thm}[products of coset spaces as parameter space]
  \label{thm:prodcosetpar}
  Let $M$ be an oriented closed connected manifold with fundamental
  group~$\Gamma$, let $S$ be a set of finite index subgroups
  of~$\Gamma$ that is stable under finite intersections. For~$\Lambda
  \in S$ we denote the covering space of~$M$ associated with the
  subgroup~$\Lambda \subset \Gamma$ by~$M_\Lambda$.
  \begin{enumerate}
  \item
    Then 
    \[ \ifsv{M}^{X_{\Gamma, S}} = \inf_{\Lambda \in S} \frac1{[\Gamma : \Lambda]} \cdot \isv{M_\Lambda} 
    \]
  \item In particular: If $S$ is the set of all finite index subgroups of~$\Gamma$, 
    then
    \[ \ifsv{M}^{X_{\Gamma, S}} = \stisv M. 
    \]
  \end{enumerate}
\end{thm}

The second part is nothing but Theorem~\ref{thm:ifsvstisvgen}.

\begin{proof}
  By definition of the stable integral simplicial volume and because
  finite intersections of finite index subgroups of~$\Gamma$ have
  finite index in~$\Gamma$, it suffices to prove the first part.

  We first show that the right hand side is an upper bound for the
  left hand side:
  For all~$\Lambda \in S$ we have 
  \[ \ifsv M ^{X_{\Gamma,S}} \leq \ifsv M^{\Gamma/\Lambda} 
                        = \frac1{[\Gamma : \Lambda]} \cdot \isv{M_\Lambda} 
  \]
  by Proposition~\ref{prop:prodpar} and
  Corollary~\ref{cor:cosetpar}. Taking the infimum yields
  \[ \ifsv{M}^{X_{\Gamma, S}} 
     \leq \inf_{\Lambda \in S} \frac1{[\Gamma : \Lambda]} \cdot \isv{M_\Lambda}.
  \]

  It remains to show that the left hand side also is an upper bound
  for the right hand side: The main idea is to reduce the parameter 
  space~$X_{\Gamma, S}$ to finite products of coset spaces. So, let $F(S)$ 
  be the set of finite subsets of~$S$. For~$F \in F(S)$ we write $\sigma_F$ 
  for the $\sigma$-algebra of the finite product~$X_{\Gamma, F}$ (which is
  just the power set of~$X_{\Gamma,F}$), we write
  \[ \pi_F \colon X_{\Gamma, S} \longrightarrow X_{\Gamma, F} 
  \]
  for the canonical projection, and we define
  \[ L_F := \bigl\{ f \circ \pi_F \bigm| f \in \linfz{X_{\Gamma,F}} \bigr\}
     \subset \linfz{X_{\Gamma,S}}.
  \]

  As first step, we will determine~$\ifsv M^{X_{\Gamma,F}}$ for~$F \in F(S)$: 
  For all~$\gamma, \gamma' \in \Gamma$ we have 
  \[ \bigl(\fa{\Lambda \in F} 
     \gamma \cdot \Lambda = \gamma' \cdot \Lambda
     \bigr)
     \Longleftrightarrow
     \gamma \cdot \bigcap_{\Lambda \in F} \Lambda 
     = \gamma' \cdot \bigcap_{\Lambda \in F} \Lambda.
  \]
  Hence, $X_{\Gamma, F} = \prod_{\Lambda \in F} \Gamma/\Lambda$ is, as
  a $\Gamma$-parameter space, a finite convex combination of multiple
  copies of the coset space~$\Gamma/\bigcap_{\Lambda \in F} \Lambda$. In 
  view of Proposition~\ref{prop:convcombpar} we obtain
  \[ \ifsv M^{X_{\Gamma, F}} = \ifsv M^{\Gamma/\bigcap_{\Lambda \in F} \Lambda}. 
  \]

  As second step, we will now show that 
  \[ L := \bigcup_{F \in F(S)} L_F \subset \linfz{X_{\Gamma, S}}
  \]
  is $\lone{\cdot}$-dense in~$\linfz{X_{\Gamma,S}}$: The
  $\sigma$-algebra~$\sigma$ on the product space~$X_{\Gamma,S}$ is the
  product $\sigma$-algebra of the power sets of all
  factors~$\Gamma/\Lambda$ with~$\Lambda \in S$; i.e., $\sigma$ is
  generated by~$\bigcup_{F \in F(S)} \pi_F^{-1}(\sigma_F)$. We now
  consider the system
  \[ \sigma' := \bigl\{ A \in \sigma \bigm| \chi_A \in \overline{L}^{\lone{\cdot}}
                \bigr\} 
  \]
  of subsets of~$X_{\Gamma, S}$. In order to show that $L$ is
  $\lone\cdot$-dense in~$\linfz{X_{\Gamma,S}}$ it suffices to prove
  that $\sigma' = \sigma$: It is easy to check that $\sigma'$ is a
  $\sigma$-algebra on~$X_{\Gamma, S}$.  %\clcomm{add details?!}
  Moreover, by definition,~$\bigcup_{F \in F(S)} \pi_F^{-1} (\sigma_F)
  \subset \sigma'$. Hence, $\sigma' = \sigma$, and so
  $L$ is $\lone{\cdot}$-dense in~$\linfz{X_{\Gamma,S}}$. 

  Therefore, the map
  \[ H_*(M,L) \longrightarrow H_*\bigl(M,\linfz{X_{\Gamma, S}}\bigr)
  \]
  induced by the inclusion $L \hookrightarrow \linfz{X_{\Gamma,S}}$ of
  coefficient $\Z \Gamma$-modules is
  isometric~\cite[Lemma~2.9]{mschmidt}\cite[Proposition~1.7]{loehphd}
  (the cited proofs carry over to this integral setting) and the
  $X_{\Gamma,S}$-fundamental class of~$M$ is contained in the image;
  notice that the union~$L = \bigcup_{F \in F(S)} L_F$ indeed is a $\Z
  \Gamma$-submodule of~$\linfz{X_{\Gamma,S}}$.

  Let $c \in C_*(M,L)$ be an $L$-fundamental cycle of~$M$. By definition
  of~$L$, there exists an~$F \in F(S)$ such that $c \in C_*(M;L_F)$. 
  Thus,
  \[ \ifsv c ^L
     \geq \ifsv M^{X_{\Gamma, F}}
     =    \ifsv M^{\Gamma/\bigcap_{\Lambda \in F} \Lambda}
     \geq \inf_{\Lambda \in S} \ifsv M^{\Gamma/\Lambda}
     = \inf_{\Lambda \in S} \frac 1{[\Gamma:\Lambda]} \cdot \isv {M_\Lambda};
  \]
  the first inequality is a consequence of the isometric isomorphism~$L_F
  \longrightarrow \linfz{X_{\Gamma,F}}$ induced from the projection~$\pi_F$, 
  the second equality was shown in the first step, the third inequality 
  holds because $S$ is assumed to be closed under finite intersections, 
  and the last equality follows from Corollary~\ref{cor:cosetpar}. 
  Taking the infimum over all fundamental cycles and taking the isometry
  $H_*(M;L) \longrightarrow H_*\bigl(M;\linfz{X_{\Gamma, S}}\bigr)$ into account 
  gives the desired estimate
  \[  \ifsv M ^{X_{\Gamma,S}} 
      \geq \inf_{\Lambda \in S}
           \frac1{[\Gamma : \Lambda]} \cdot \isv{M_\Lambda}.
      \qedhere
  \]
\end{proof}

\begin{rem}[inverse limits]
  The proof of Theorem~\ref{thm:prodcosetpar} carries over to the
  following modification: Instead of the product~$\prod_{\Lambda \in
    S} \Gamma/\Lambda$ we can also consider the inverse limit~$I_{\Gamma,S}$ of
  the system~$(\Gamma/\Lambda)_{\Lambda \in S}$ with respect to the
  canonical projection maps between coset spaces of nested subgroups; we 
  equip~$I_{\Gamma,S}$ with the $\Gamma$-action induced from the translation action 
  on the coset spaces and we equip~$I_{\Gamma,S}$ with the Borel probability space 
  structure given by the discrete $\sigma$-algebras on the coset spaces and the 
  compatible system of normalised counting measures on the coset spaces. 
  We then obtain
  \[  \ifsv M ^{I_{\Gamma,S}} = \inf_{\Lambda \in S} \frac1{[\Gamma:\Lambda]} \cdot \isv {M_\Lambda}.
  \]
  While this inverse limit space might be harder to visualise than the 
  product space~$X_{\Gamma,S}$, it does have the advantage that it is ergodic.
  % look at projections to inverse limits over finite subsystems with 
  % a minimal element. These are easy to describe explicitly, and \Gamma-invariant
  % subsets are either empty or everything. 
\end{rem}

%%%%%%%%%%%%%%%%%%%%%%%%%%%%%%%%%%%%%%%%%%%%%%%%%%%%%%%%%%%%%%%%%%%%%%%%%%%%%%%%
\section{Integral foliated simplicial volume\\ of hyperbolic $3$-manifolds}\label{sec:mainthm}

We will now complete the proof of Theorem~\ref{mainthm} by proving the 
following, more explicit, version:

\begin{thm}[integral foliated simplicial volume of hyperbolic $3$-manifolds]
  Let $M$ be an oriented closed connected hyperbolic $3$-manifold, and
  let $S$ be a set of representatives of uniform lattices in~$G :=
  \Isom^+(\Hyp^3)$ containing a representative for every isometry
  class of oriented closed connected hyperbolic
  $3$-manifolds (see Theorem~\ref{thm:ifsvstisvhyppar}). 
  Then 
  \[ \ifsv M = \ifsv M ^{\prod_{\Lambda \in S} G/\Lambda} = \sv M. 
  \]
\end{thm}

\begin{proof}
  Let $M$ be an oriented closed connected hyperbolic $3$-manifold. In
  view of Theorem~\ref{thm:hyp3stisv} there exists a
  sequence~$(M_n)_{n \in \N}$ of oriented closed connected hyperbolic
  $3$-manifolds with
  \[ \lim_{n \rightarrow \infty} \frac{\stisv{M_n}}{\sv {M_n}} = 1. 
  \]
  By Corollary~\ref{cor:ifsvstisvhyp} (and
  Theorem~\ref{thm:ifsvstisvhyppar}, for the concrete parameter
  space), for all~$n\in\N$ we have
  \[ \sv M
     \leq \ifsv M \leq \ifsv M ^{\prod_{\Lambda \in S} G/\Lambda} 
     \leq \frac{\vol M}{\vol M_n} \cdot \stisv{M_n}.
  \]
  On the other hand, by the classical proportionality principle for 
  simplicial volume for hyperbolic manifolds 
  we have $\sv{M_n} > 0$ and 
  \[ \frac{\vol M}{\vol M_n} = \frac{\sv M}{\sv {M_n}}, 
  \]
  and so
  \[ \sv M \leq \ifsv M \leq \ifsv M ^{\prod_{\Lambda \in S} G/\Lambda} \leq \frac{\sv M}{\sv {M_n}} \cdot \stisv{M_n}. 
  \]
  Because of $\lim_{n \rightarrow \infty}\stisv{M_n}/\sv{M_n} = 1$ the
  right hand side converges to~$\sv M$ for~$n \rightarrow \infty$. Hence, 
  $\sv M = \ifsv M ^{\prod_{\Lambda \in S} G/\Lambda} = \ifsv M$, as desired.
\end{proof}

Similarly to Francaviglia, Frigerio, and Martelli~\cite[Question~6.4]{ffm}, 
we hence ask:

\begin{question}
  Does the integral foliated simplicial volume of oriented closed
  connected hyperbolic manifolds of dimension bigger than~$3$ also
  coincide with the simplicial volume? Does this even hold for all 
  oriented closed connected aspherical manifolds? 
\end{question}

%%%%%%%%%%%%%%%%%%%%%%%%%%%%%%%%%%%%%%%%%%%%%%%%%%%%
\section{Stable integral simplicial volume and integral foliated 
simplicial volume of Seifert $3$-manifolds}
\label{sec:stisvifsv3mnf}

For the sake of completeness, we add also the computation of stable 
integral simplicial volume and integral foliated simplicial volume of 
Seifert $3$-manifolds:

\begin{prop}[Seifert case]
Let $M$ be an oriented compact connected Seifert manifold with 
$|\pi_1(M)|=\infty$. Then $\stisv{M}=0$.
\end{prop}
\begin{proof}
We follow step by step the related argument of Francaviglia, Frigerio, and
Martelli~\cite[Proposition~5.11]{ffm}.  A Seifert
manifold has a finite covering that is an $S^1$-bundle over an orientable
surface~$\Sigma$ with some Euler number~$e\geq 0$. If the manifold~$M$
has boundary, then $e=0$ and the bundle is a product $S^1\times
\Sigma$.  Since $S^1 \times \Sigma$ covers itself with arbitrarily
high degree, $S^1 \times \Sigma$ (and hence also~$M$) clearly has
stable integral simplicial volume~$0$.

If $M$ is closed, we denote the covering mentioned above by $(\Sigma,e)$.
A closed connected orientable $S^1$-bundle over an 
orientable surface either is irreducible or $S^1\times S^2$
or $\mathbb{RP}^3\#\mathbb{RP}^3$.
Since $\mathbb{RP}^3\#\mathbb{RP}^3$ admits a double self-covering, in this case 
the stable integral simplicial volume vanishes. Moreover, 
we have already considered the case $S^1\times \Sigma$. 
Therefore we restrict to irreducible $S^1$-bundles.
By assumption, $(\Sigma,e)$ is a closed connected orientable 
irreducible Seifert manifold 
with $|\pi_1((\Sigma,e))|=\infty$. Thus, by Remark~\ref{complexity}, 
the complexity 
is equal to the special complexity and it has the following bound~\cite{MaPe}:
$$c_S(\Sigma,e)\leq e+6 \cdot \chi_-(\Sigma)+6$$
where $\chi_-(\Sigma)=\max\{-\chi(\Sigma),0\}$.

Let now $T$ be a minimal triangulation of $(\Sigma,e)$, and
let $T'$ be its first barycentric subdivision. Then $T'$ is a semi-simplicial 
triangulation that allows us to define an integral fundamental cycle.
Since the special complexity of $(\Sigma,e)$
is equal to the minimal number of tetrahedra
in a triangulation 
and the barycentric subdivision of~$\Delta^3$ consists 
of $24$~tetrahedra, we have 
$$\isv{(\Sigma,e)}\leq 24 \cdot c_S(\Sigma,e)\leq 24 \cdot \bigl(
e+6\cdot \chi_-(\Sigma)+6\bigr).$$ Now for every $d\in \N$ we
construct a degree-$d^2$ covering $(\overline \Sigma,  e)
\longrightarrow (\Sigma,e)$ where $\overline \Sigma$ is a $d$-sheeted
covering space of~$\Sigma$ as in the argument of Francaviglia,
Frigerio, and Martelli~\cite[Proposition~5.11]{ffm} in order to
conclude.
\end{proof}

In particular, we obtain:

\begin{cor}
Let $M$ be an oriented closed connected Seifert manifold with 
$|\pi_1(M)|=\infty$. Then 
\[ \sv M = \ifsv M = \stisv M =0. 
\]
\end{cor}

By Perelman's result a closed Seifert manifold with 
finite fundamental group admits an elliptic structure (i.e., it 
is a quotient of~$S^3$ by a finite subgroup of~$\mathrm{SO}(4)$ acting 
by rotations). 

\begin{prop}
  Let $n \in \N_{>0}$, and let $M$ be an oriented closed connected
  elliptic $n$-manifold. Then
$$\ifsv{M}=\stisv{M}= \frac1{|\pi_1(M)|} \cdot 
\begin{cases}
  1 & \text{if $n$ is odd}\\
  2 & \text{if $n$ is even.}
\end{cases}$$
In contrast, $\sv{M}=0$.
\end{prop}
\begin{proof}
  Corollary~\ref{cor:finitepi1} shows that
  \[ \ifsv M = \stisv M = \frac1{|\pi_1(M)|} \cdot \isv[norm]{S^n}. 
  \]
  Clearly, $\isv[norm]{S^n} =1$ if $n$ is odd, and $\isv[norm]{S^n}=2$ 
  if $n$ is even.

On the other hand,the simplicial volume of an elliptic manifold is always zero
because simplicial volume is multiplicative under finite coverings 
\cite[Proposition 4.1]{loeh} and $\sv{S^n} = 0$ 
  for all~$n\in \N_{>0}$.
\end{proof}

\end{document}